\documentclass[11pt,a4paper]{amsart}
\usepackage[english]{babel}
\usepackage{pst-grad} % For gradients
\usepackage{pst-plot} % For axes
\usepackage{pstricks}
\usepackage{amsmath,amssymb,mathrsfs,amsthm}
\usepackage[dvips]{graphicx}
\usepackage{epsfig}
\usepackage{color}
\newcommand{\RR}{\mathbb R}
\newcommand{\BB}{\mathbb B}
\newcommand{\VV}{\mathbb V}
\newcommand{\NN}{\mathbb N}
\newcommand{\ZZ}{\mathbb Z}

\newcommand{\TT}{\mathbb T}
\newcommand{\pat}{\partial_t}
\newcommand{\pax}{\partial_x}

\newcommand{\jeps}{\mathcal{J}_\vartheta*}
\newcommand{\jpi}{\mathcal{J}_\varpi*}

\newcommand{\pv}{\text{P.V.}\int_\RR}
\newcommand{\ueps}{u_\vartheta}
\newcommand{\upi}{u_\varpi}

\newcommand{\card}{\mathop{\mathrm{card}}\nolimits}

\DeclareMathOperator*{\sgn}{sgn}

\newcounter{comentcount}
\newcounter{teocount}
\newtheorem{lem}{Lemma}

\newtheorem{corol}{Corollary}
\newtheorem{teo}[teocount]{Theorem}
\newtheorem{defi}{Definition}

\title[A nonlocal KS equation]{On a nonlocal analog of the Kuramoto-Sivashinsky equation}
\author[R.Granero-Belinch\'{o}n]{Rafael Granero-Belinch\'{o}n}
\email{rgranero@math.ucdavis.edu}
\address{Department of Mathematics, University of California, Davis, CA 95616, USA}
\author[J. Hunter]{John K.\ Hunter}
\thanks{The first author receives financial support by the grant MTM2011-26696 from the former Ministerio de Ciencia e Innovaci\'on (MICINN, Spain). The second author was partially supported by the NSF under grant number DMS-1312342.}
\email{jkhunter@ucdavis.edu}
\address{Department of Mathematics, University of California, Davis, CA 95616, USA}
\begin{document}

\begin{abstract}
We study a nonlocal equation, analogous to the Kuramoto-Sivashinsky
equation, in which short waves are stabilized by a possibly fractional diffusion of order less than or equal to two, and long waves are destabilized by a backward fractional diffusion of lower order. We prove the global existence, uniqueness, and analyticity of solutions of the nonlocal equation and the existence of a compact attractor. Numerical results show that the equation has chaotic solutions whose spatial structure consists of interacting traveling waves resembling viscous shock profiles.
\end{abstract}

\maketitle


%\textbf{Acknowledgments}:

%\tableofcontents

\section{Introduction}
In this paper, we study a family of nonlinear, nonlocal pseudo-differential equations
in one-space dimension for a function $u(x,t)$ given by
\begin{equation}\label{KS}
\pat u+\pax\left(\frac{1}{2} u^2\right)= \Lambda^\gamma u -\epsilon\Lambda^{1+\delta}u,
\end{equation}
where $\epsilon > 0$ and $\Lambda^s$ is the fractional derivative
\[
\Lambda^s=\left(-\pax^2\right)^{s/2}, \qquad \widehat{\Lambda^s u}=|\xi|^s\hat{u}.
\]
We assume that the exponents $\delta$, $\gamma$ satisfy
\begin{equation}\label{gammadelta}
0 < \delta \leq 1,\qquad 0 \le \gamma<1+\delta.
\end{equation}

Equation \eqref{KS} consists of an inviscid Burgers equation with a
higher-order linear pseudo-differential term that gives long-wave instability and short-wave stability. It is analogous to the well-known Kuramoto-Sivashinsky (KS) equation \cite{kuramoto1976persistent, sivashinsky1977nonlinear, sivashinsky1980flame}
\begin{equation}\label{usualKS}
\pat u+ \pax\left(\frac{1}{2} u^2\right)=-\pax^2 u-\epsilon\pax^4 u,
\end{equation}
which has negative second-order diffusion stabilized by forth-order diffusion.
By contrast, we consider \eqref{KS} in the parameter regime \eqref{gammadelta}, where the stabilizing diffusion is second-order or less.

A special case of \eqref{KS}, corresponding to $\gamma = \delta = 1$, is
\begin{equation}\label{KS11}
\pat u+\pax\left(\frac{1}{2} u^2\right)= \Lambda u + \epsilon\pax^2 u,
\end{equation}
which provides a simple model for the stabilization of a Hadamard instability, with growth rate
proportional to the absolute value of the wavenumber,
by second-order viscous diffusion. This type of instability occurs in scale-invariant systems,
such as conservation laws (e.g., the Kelvin-Helmholtz instability
for the Euler or MHD equations) and kinetic equations (e.g., the Vlasov equations), in which the growth rate of long waves is determined by a parameter with the dimensions of velocity.
In particular, \eqref{KS11} provides a model equation for the
negative Landau damping of plasma waves \cite{lee,ott}.

If $\gamma = 0$ and $\delta = 1$, then \eqref{KS}
is the Burgers-Sivashinsky (BS) equation introduced
by Goodman \cite{goodmankuramoto},
\begin{equation}\label{usualBS}
\pat u+\pax\left(\frac{1}{2}u^2\right) = u+\epsilon\pax^2 u.
\end{equation}
For \eqref{usualBS}, the growth rate of long waves is bounded independently of the wavenumber, and its
dynamical behavior is much simpler than that of \eqref{KS} with $\gamma > 0$.

The KS equation \eqref{usualKS} exhibits chaotic behavior and possesses a compact global attractor \cite{NSTkuramoto, NSTkuramoto2}. Furthermore, it has an inertial manifold \cite{FNSTkuramoto} that appears to contain a chaotic attractor when $\epsilon$ is sufficiently small.
(See \cite{BronskiKuramoto, colletkuramoto2, goodmankuramoto, giacomelliottokuramoto, ottokuramoto} for further results).
The spatial analyticity of solutions of the KS equation is addressed in \cite{colletkuramoto1, AnalyticityKuramotoGrujic} and the temporal analyticity in \cite{timeanalyticityKuramoto}. More recently, the authors in \cite{ArioliKochKuramoto, Figueras-DeLaLLave:cap-periodic-orbits-kuramoto, Mischaikowrigorous} have used computer-assisted methods to study the dynamics of the solutions.

In this paper, we prove that \eqref{KS} possesses a compact global attractor in the parameter range \eqref{gammadelta} (see Theorem~\ref{attractor}).
Moreover, numerical solutions indicate that if $0<\gamma <1+\delta$, then \eqref{KS} exhibits chaotic behavior with an interesting spatial structure. Waves that resemble thin viscous shocks appear spontaneously at different points, after which they propagate toward and merge with a primary viscous shock. This spatial behavior is qualitatively different from what one sees in the usual KS equation.
(See Section~\ref{sec:numerics}.)
By contrast, solutions of the BS equation \eqref{usualBS}, with $\gamma=0$, do not behave chaotically; instead, they approach a
time-independent viscous sawtooth wave solution as $t\to \infty$  \cite{goodmankuramoto}.

The numerical results suggest that \eqref{KS} with exponents \eqref{gammadelta} may have an inertial manifold that can be parametrized in some way by the viscous shocks. We do not investigate this question here, but in Section~\ref{sec:wild} we obtain an upper bound
on the number of oscillations in solutions of \eqref{KS} (see Theorem~\ref{th:wild_oscillations}).

Nonlocal KS equations similar to \eqref{KS} have been studied previously by Frankel and Roytburd \cite{frankelroytburd}.
Their results, however, are less detailed than ours and they apply only in the case when $\delta \ge 1$.
A different type of nonlocal generalization of the KS equation has been studied in
\cite{NonlocalKuramotoBronski, NonlocalKuramotoDuan}.

We conclude the introduction by outlining the contents of this paper. In Section~\ref{sec:global}, we prove the global existence of smooth solutions of \eqref{KS}, and in Section~\ref{sec:analyticity}, we prove that these solutions gain analyticity in a strip. In Sections~\ref{sec:absorbing}--\ref{sec:attractor}, we prove the existence of an attractor for \eqref{KS}, and in Section~\ref{sec:numerics}, we show some numerical solutions.

\section{Global existence of solutions}
\label{sec:global}
In this section we use a classical energy method to prove the global existence of solutions of the initial value problem for \eqref{KS},
\begin{align}\label{KSivp}
\begin{split}
&\pat u+\pax\left(\frac{1}{2} u^2\right)= \Lambda^\gamma u -\epsilon\Lambda^{1+\delta}u,\qquad \mbox{$x\in\Omega$, $t>0$},
\\
& u(x,0) = u_0(x),\qquad x\in \Omega.
\end{split}
\end{align}
We consider either spatially periodic solutions or solutions on the real line, with $\Omega=\TT$ or $\Omega=\RR$ as appropriate.
In the periodic case, we normalize the length of $\TT$ to $2\pi$.

To prove the existence result, we first obtain an \emph{a priori} $L^\infty$-estimate, using the ideas in \cite{cor2} to handle the nonlocal operators (see also \cite{AGM, c-g09}). This step of the proof depends on the choice of $\Omega$ and $\delta$ and is different in each case. In Section~\ref{sec:exist0} we obtain the existence of solutions for $0<\delta<1$. The gain of derivatives can be as small as $1/2+\delta/2$,
%and $\delta$ may be arbitrarily close to zero,
so the well-posedness results are more delicate than for the usual KS or BS equations. In Section~\ref{sec:exist1} we treat the simpler case $\delta=1$. To simplify the notation, we omit the $t$-dependence of $u$ when convenient and use $C$ to denote a (harmless) constant that can change from one line to another

First, we define what we mean by a weak solution of \eqref{KSivp}.
We denote the usual Sobolev spaces of functions with weak $L^2$-derivatives of the order less than or equal to $s$ by $H^s(\Omega)$, or $H^s$, and the real or periodic spatial Hilbert transform, with symbol $-i\sgn \xi$, by $\mathcal{H}$. In particular, $\Lambda = \mathcal{H}\partial_x$.

\begin{defi}\label{defi2b}
Let $T > 0$. A function $u(x,t)$ with
$$
u(x,t)\in L^2([0,T],H^{\frac{1+\delta}{2}}), \qquad\pat u(x,t)\in L^2([0,T],H^{-\frac{1+\delta}{2}})
$$
is a weak solution of \eqref{KSivp} if the following equality holds for all test functions $\phi\in H^{\frac{1+\delta}{2}}(\Omega)$,
\begin{align*}
&\int_\Omega\phi\, \pat u\, dx - \frac{1}{2}\int_\Omega \Lambda^{\frac{1+\delta}{2}}\phi\, \Lambda^{1-\frac{1+\delta}{2}} \mathcal{H}(u^2)\, dx\\
&\qquad
=\int_\Omega\Lambda^{\gamma/2}\phi\, \Lambda^{\gamma/2}u\, dx-\epsilon \int_\Omega\Lambda^{(1+\delta)/2}\phi\, \Lambda^{(1+\delta)/2}u\, dx
\quad\mbox{a.e.\ $0<t<T$},
\end{align*}
and $u(x,0)=u_0(x)$.
\end{defi}

%Let us mention a classical inequality concerning the boundedness of the Hilbert transform:
%\begin{equation}\label{BoundH}
%\|\mathcal{H} f\|_{L^p}\leq c_p\|f\|_{L^p}\,\text{ for }1<p<\infty.
%\end{equation}
%Notice that, as an application of \eqref{BoundH}, we have that
%$$
%\|\Lambda^{1-\frac{1+\delta}{2}} \mathcal{H}(u^2)\|_{L^2}\leq C
%\|\Lambda^{1-\frac{1+\delta}{2}}(u^2)\|_{L^2}\leq C\|u^2\|_{H^{\frac{1-\delta}{2}}}\leq C\|u\|_{H^{\frac{1-\delta}{2}}}C\|u\|_{L^\infty}
%$$
We remark that the $L^2$-boundedness of $\mathcal{H}$, a Moser-type inequality \cite{taylor}, and Sobolev inequalities, imply that
\begin{align*}
\|\Lambda^{1-\frac{1+\delta}{2}} \mathcal{H}(u^2)\|_{L^2} &\le
\|\Lambda^{\frac{1-\delta}{2}}(u^2)\|_{L^2}\leq C\|u^2\|_{H^{\frac{1-\delta}{2}}}
\leq C\|u\|_{L^\infty} \|u\|_{H^{\frac{1-\delta}{2}}}
\\
&
\le  C\|u\|_{H^{\frac{1+\delta}{2}}} \|u\|_{H^{\frac{1-\delta}{2}}},
\end{align*}
so the nonlinear term in this weak formulation is well-defined.

\subsection{The case $0<\delta<1$}
\label{sec:exist0}
First, we consider spatially periodic solutions.
Since the mean of $u$ is preserved by the evolution, we can restrict ourselves to periodic initial data with zero mean,
\begin{equation*}%\label{hip1}
\int_\TT u_0(x)dx=0.
\end{equation*}

\begin{lem}\label{Linftytoro}
If $u(x,t)$ is a spatially periodic, smooth solution of \eqref{KSivp}, then
$$
\|u(t)\|_{L^\infty(\RR)}\leq \|u_0\|_{L^\infty(\RR)}\exp\left(C(\epsilon,\gamma,\delta) t\right).
$$
\end{lem}
\begin{proof}
The fractional derivatives can be written as \cite{AGM}
\begin{align}\label{periodickernel}
\begin{split}
\Lambda^\alpha u(x)&=\frac{\Gamma(1+\alpha)\cos\left((1-\alpha)\frac{\pi}{2}\right)}{\pi}\pv \frac{u(x)-u(\eta)}{|x-\eta|^{1+\alpha}}d\eta\\
&=\frac{\Gamma(1+\alpha)\cos\left((1-\alpha)\frac{\pi}{2}\right)}{\pi}\sum_{k\in\ZZ}\text{P.V.}\int_\TT \frac{u(x)-u(\eta)}{|x-\eta-2k\pi|^{1+\alpha}}d\eta\\
&=\frac{\Gamma(1+\alpha)\cos\left((1-\alpha)\frac{\pi}{2}\right)}{\pi}\sum_{k\in\ZZ}\text{P.V.}\int_\TT \frac{u(x)-u(x-\eta)}{|\eta-2k\pi|^{1+\alpha}}d\eta
\end{split}
\end{align}
and
$$
\Lambda u(x)=\frac{1}{2\pi}\text{P.V.}\int_\TT \frac{u(x)-u(x-\eta)}{\sin^2\left(\frac{\eta}{2}\right)}d\eta.
$$

We start the proof with the case $\gamma=1$, for which we have a concise expression for the  kernel.
Let $x_t$ denote the point where $u(\cdot,t)$ attains its maximum, and suppose that the $L^\infty$-norm $\|u(t)\|_{L^\infty(\RR)}=u(x_t)$ is achieved at the maximum of $u$. A straightforward calculation shows that $u(x_t)$ is a Lipschitz continuous function of $t$, so Rademacher's Theorem \cite{evans} implies that $\|u(t)\|_{L^\infty(\RR)}$ is differentiable pointwise almost everywhere. Now we can apply the technique developed in \cite{AGM, cor2, c-g09, CGO}, to obtain the evolution of ${du(x_t)}/{dt}$. Using the expressions for the kernels, we get
\begin{align*}
&\frac{d}{dt}\|u(t)\|_{L^\infty(\RR)}\leq\frac{1}{2\pi}\text{P.V.}\int_\TT \left(u(x_t)-u(x_t-\eta)\right)\left(\frac{1}{\sin^2\left(\frac{\eta}{2}\right)}-\frac{1}{\left(\frac{\eta}{2}\right)^2}\right)d\eta\\
&\qquad+\frac{1}{2\pi}\text{P.V.}\int_\TT \left(u(x_t)-u(x_t-\eta)\right)\left(\frac{1}{\left(\frac{\eta}{2}\right)^2}
-\frac{2\epsilon\Gamma(2+\delta)\cos\left(\delta\frac{\pi}{2}\right)}{|\eta|^{2+\delta}}\right)d\eta.
\end{align*}
The first term is not singular and can be estimated as follows:
\begin{align*}
&I_1=\frac{1}{2\pi}\text{P.V.}\int_\TT \left(u(x_t)-u(x_t-\eta)\right)\left(\frac{1}{\sin^2\left(\frac{\eta}{2}\right)}-\frac{1}{\left(\frac{\eta}{2}\right)^2}\right)d\eta\\
&\qquad\qquad\leq \frac{2\|u(t)\|_{L^\infty(\TT)}}{\pi} \int_0^\pi \left(\frac{1}{\sin^2\left(\frac{\eta}{2}\right)}-\frac{1}{\left(\frac{\eta}{2}\right)^2}\right)d\eta
\\
&\qquad\qquad\leq \frac{8}{\pi^2}\|u(t)\|_{L^\infty(\TT)}.
\end{align*}
Notice that there exists $\omega=\omega(\delta,\epsilon)$ such that for $0<|\eta|\leq \omega$, we have
$$
\left(u(x_t)-u(x_t-\eta)\right)\left(\frac{1}{\left(\frac{\eta}{2}\right)^2}
-\frac{2\epsilon\Gamma(2+\delta)\cos\left(\delta\frac{\pi}{2}\right)}{|\eta|^{2+\delta}}\right)\leq 0
$$
We split the second term as
\begin{eqnarray*}
I_2&=&\frac{\text{P.V.}}{2\pi}\int_\TT \left(u(x_t)-u(x_t-\eta)\right)\left(\frac{1}{\left(\frac{\eta}{2}\right)^2}
-\frac{2\epsilon\Gamma(2+\delta)\cos\left(\delta\frac{\pi}{2}\right)}{|\eta|^{2+\delta}}\right)d\eta\\
&\leq& J_1+J_2
\end{eqnarray*}
with
\begin{eqnarray*}
J_1&=&\frac{\text{P.V.}}{2\pi}\int_{B(0,\omega)} \left(u(x_t)-u(x_t-\eta)\right)\left(\frac{1}{\left(\frac{\eta}{2}\right)^2}
-\frac{2\epsilon\Gamma(2+\delta)\cos\left(\delta\frac{\pi}{2}\right)}{|\eta|^{2+\delta}}\right)d\eta\\
&\leq& 0,
\end{eqnarray*}
and
\begin{eqnarray*}
J_2&=&\frac{1}{2\pi}\int_{B^c(0,\omega)} \left(u(x_t)-u(x_t-\eta)\right)\left(\frac{1}{\left(\frac{\eta}{2}\right)^2}
-\frac{2\epsilon\Gamma(2+\delta)\cos\left(\delta\frac{\pi}{2}\right)}{|\eta|^{2+\delta}}\right)d\eta\\
&\leq& C(\epsilon,\delta)\|u(t)\|_{L^\infty(\TT)},
\end{eqnarray*}
thus,
$$
I_2=J_1+J_2\leq J_2\leq C(\epsilon,\delta)\|u(t)\|_{L^\infty(\TT)}.
$$
The same argument applies if $\|u(t)\|_{L^\infty(\RR)}=-\min_{x\in \TT} u(x,t)$, so
$$
\|u(t)\|_{L^\infty(\RR)}\leq \|u_0\|_{L^\infty(\RR)}\exp\left(C(\epsilon,\delta) t\right).
$$

In the general case $\gamma\neq1$, some extra terms appear. These terms correspond to $|k|\geq1$ in \eqref{periodickernel}. Since they are not singular, they can be estimated as follows:
\begin{align*}
&\frac{\Gamma(1+\gamma)\cos\left((1-\gamma)\frac{\pi}{2}\right)}{\pi}\sum_{|k|\geq1}\text{P.V.}\int_\TT \frac{u(x_t)-u(x_t-\eta)}{|\eta-2k\pi|^{1+\gamma}}d\eta\\
&\qquad\qquad\leq C(\gamma)\|u(t)\|_{L^\infty(\TT)}.
\end{align*}
The rest of the proof remains unchanged.
\end{proof}

Next, we prove our main existence result.

\begin{teo}\label{teotoro}
Suppose that $\epsilon > 0$, $0 < \delta < 1$, and $0 \le \gamma<1+\delta$. If
\[
u_0\in H^\alpha(\TT)\cap L^\infty(\TT),
\]
then the following statements hold:
\begin{itemize}
\item If $\alpha\geq 2+\delta$, then for every $0<T<\infty$ the initial value problem \eqref{KSivp} has a unique classical solution
\[
u(x,t)\in C([0,T],H^\alpha(\TT)).
\]
\item If $(1-\delta)/2 < \alpha<2+\delta$, then for every $0< T <\infty$ there exists a weak solution of \eqref{KSivp} (see Definition \ref{defi2b}) such that
$$
u(x,t)\in L^\infty([0,T],H^\alpha(\TT)\cap L^\infty(\TT))\cap C([0,T],H^s(\TT)\cap L^p(\TT))
$$
for every $0\leq s<\alpha$ and $2\leq p<\infty$.
\item These solutions gain regularity and satisfy
$$
u(x,t)\in L^2([0,T],H^{\alpha+\frac{1+\delta}{2}}(\TT)).
$$
Moreover, if $3/2 < \alpha+(1+\delta)/2$, then this weak solution is unique.
\end{itemize}

\end{teo}
\begin{proof}
\textbf{Step 1: $L^2$ estimate.} We multiply \eqref{KS} by $u$ and integrate by parts:
$$
\frac{1}{2}\frac{d}{dt}\|u\|_{L^2}^2=-\frac{\epsilon}{2}\|\Lambda^{\frac{1+\delta}{2}}u\|_{L^2}^2+\int_\TT u(x)\left(\Lambda^\gamma-\frac{\epsilon}{2}\Lambda^{1+\delta}\right)u \,dx.
$$
Using the Fourier transform, we get
$$
\int_\TT u(x)\left(\Lambda^\gamma- \frac{\epsilon}{2}\Lambda^{1+\delta}\right)u \,dx\leq \left(\frac{2\gamma}{\epsilon(1+\delta)}\right)^{1/(1+\delta-\gamma)}\|u(t)\|_{L^2(\TT)}^2.
$$
Inserting this into the previous bound we obtain
$$
\frac{d}{dt}\|u\|_{L^2}^2\leq -\epsilon\|\Lambda^{\frac{1+\delta}{2}}u\|_{L^2}^2+2 \left(\frac{2\gamma}{\epsilon(1+\delta)}\right)^{1/(1+\delta-\gamma)}\|u(t)\|_{L^2(\TT)}^2,
$$
and using Gronwall inequality,
\begin{multline}\label{l2}
\|u(t)\|_{L^2(\TT)}^2+\epsilon\int_0^t \exp\left(2\left(\frac{2\gamma}{\epsilon(1+\delta)}\right)^{1/(1+\delta-\gamma)}(t-s)\right)\|\Lambda^{\frac{1+\delta}{2}}u(s)\|_{L^2}^2\, ds\\
\leq \|u_0\|_{L^2(\RR)}^2\exp\left(2\left(\frac{2\gamma}{\epsilon(1+\delta)}\right)^{1/(1+\delta-\gamma)}t\right).
\end{multline}
In particular
\begin{multline*}
\|u(t)\|_{L^2(\TT)}+\epsilon\int_0^t \|\Lambda^{\frac{1+\delta}{2}}u(s)\|_{L^2}^2\, ds\\
\leq \|u_0\|_{L^2(\TT)}^2\exp\left(2\left(\frac{2\gamma}{\epsilon(1+\delta)}\right)^{1/(1+\delta-\gamma)}t\right).
\end{multline*}

\textbf{Step 2: $H^\alpha$ estimate.} We multiply \eqref{KS} by $\Lambda^{2\alpha}u$ and integrate, which gives
$$
\frac{d}{dt}\|\Lambda^\alpha u\|_{L^2(\TT)}^2=I_1+I_2+I_3,
$$
where
\begin{align*}
I_1&=\int_\TT \Lambda^{\alpha+\frac{1+\delta}{2}} u\Lambda^{\alpha+1-\frac{1-\delta}{2}}\mathcal{H}(u^2)\,dx,
\\
I_2&=2\int_\TT \Lambda^\alpha u\left(\Lambda^\gamma-\frac{\epsilon}{2}\Lambda^{1+\delta}\right)\Lambda^\alpha u\,dx
\leq C(\epsilon,\gamma,\delta) \|\Lambda^\alpha u\|_{L^2(\TT)}^2,
\\
I_3&=-\epsilon\|\Lambda^{\alpha+\frac{1+\delta}{2}}u\|_{L^2(\TT)}^2.
\end{align*}
The term $I_1$ can be handled as follows (see also \cite{kiselevburgers}): We use the Cauchy-Schwarz and Kato-Ponce inequalities (see Lemma \ref{commutator}) and the properties of the Hilbert transform (see \cite{stein1970singular}) to get
\begin{align*}
I_1&\leq \|\Lambda ^ {\alpha+\frac{1+\delta}{2}}u\|_{L^2(\TT)}\|\Lambda^{\alpha+1-\frac{1+\delta}{2}}(u^2)\|_{L^2(\TT)}\\
&\leq C\|\Lambda ^ {\alpha+\frac{1+\delta}{2}}u\|_{L^2(\TT)}\|\Lambda ^ {\alpha+1-\frac{1+\delta}{2}}u\|_{L^2(\TT)}\|u\|_{L^\infty(\TT)}.
\end{align*}
Then, using
$$
\alpha+1-\frac{1+\delta}{2}=t(\alpha+\frac{1+\delta}{2})+(1-t)\alpha,
$$
for $t=-1+2/(1+\delta)$, and H\"older's inequality on the Fourier side (with $p=1/t$ and $q=1/(1-t)$), we write
$$
\|\Lambda ^ {\alpha+1-\frac{1+\delta}{2}}u\|_{L^2(\TT)}^2\leq\|\Lambda ^ {\alpha+\frac{1+\delta}{2}}u\|_{L^2(\TT)}^{2t}\|\Lambda ^ {\alpha}u\|_{L^2(\TT)}^{2(1-t)}.
$$
Inserting this into the bound for $I_1$, we obtain
$$
I_1\leq C\|\Lambda ^ {\alpha+\frac{1+\delta}{2}}u\|_{L^2(\TT)}^{1+t}\|\Lambda ^ {\alpha}u\|_{L^2(\TT)}^{1-t}\|u\|_{L^\infty(\TT)}.
$$
Using H\"older's inequality again (with $p=2/(1+t)$ and $q=2/(1-t)$), we get
$$
I_1\leq C(\epsilon,\delta)\|\Lambda ^ {\alpha}u\|_{L^2(\TT)}^{2}\|u\|_{L^\infty(\TT)}^{2/(1-t)}+\frac{\epsilon}{2}\|\Lambda ^ {\alpha+\frac{1+\delta}{2}}u\|_{L^2(\TT)}^{2}.
$$
Using the estimate for $\|u(t)\|_{L^\infty(\TT)}$ and putting all the estimates together, we obtain
\begin{multline}\label{Halpha}
\frac{d}{dt}\|\Lambda^\alpha u\|_{L^2(\TT)}^2\leq \|\Lambda ^ {\alpha}u\|_{L^2(\TT)}^{2}\exp\left(C\left(\epsilon,\gamma,\delta,\|u_0\|_{L^\infty(\TT)}\right)(1+t)\right)\\
-\frac{\epsilon}{2}\|\Lambda ^ {\alpha+\frac{1+\delta}{2}}u\|_{L^2(\TT)}^{2}.
\end{multline}
Finally, from Gronwall inequality, we conclude that
\begin{multline*}
\|\Lambda^\alpha u(t)\|_{L^2(\TT)}^2+\frac{\epsilon}{2}\int_0^te^{e^{C\left(\epsilon,\gamma,\delta,\|u_0\|_{L^\infty(\TT)}\right)(1+t-s)}}\|\Lambda ^ {\alpha+\frac{1+\delta}{2}}u(s)\|_{L^2(\TT)}^{2}ds\\
\leq \|\Lambda^\alpha u_0\|_{L^2(\TT)}^2\exp\left(\exp\left(C\left(\epsilon,\gamma,\delta,\|u_0\|_{L^\infty(\TT)}\right)(1+t)\right)\right).
\end{multline*}
In particular,
$$
\int_0^t\|\Lambda ^ {\alpha+\frac{1+\delta}{2}}u\|_{L^2(\TT)}^{2}ds\leq \frac{2}{\epsilon}\|\Lambda^\alpha u_0\|_{L^2(\TT)}^2e^{e^{C\left(\epsilon,\gamma,\delta,\|u_0\|_{L^\infty(\TT)}\right)(1+t)}}.
$$

\textbf{Step 3: Strong solutions.} We denote by $\mathcal{J_\epsilon}$ a positive, symmetric mollifier. Then, in the case $\alpha>2+\delta$, we define the regularized problems
\begin{equation}\label{KSreg}
\pat \ueps+\jeps\frac{\pax(\left(\jeps\ueps\right)^2)}{2}=\jeps\left(\Lambda^\gamma -\epsilon\Lambda^{1+\delta}\right)\jeps\ueps,%\text{ }x\in\TT, t>0,
\end{equation}
with initial data
$$
\ueps(0)= u_0.
$$
By Picard's Theorem, these regularized problems have a unique solution $\ueps\in C^1([0,T],H^\alpha(\TT))$. Moreover, since the \emph{a priori} estimates remain valid, these solutions are global in time. Thus, for every $T>0$ there exists
\[
u(x,t)\in L^\infty\left([0,T],H^\alpha(\TT)\right)
\]
such that (after picking a subsequence)
\[
\ueps\rightharpoonup u\qquad \mbox{in}\quad L^2\left([0,T],H^{\alpha+\frac{1+\delta}{2}}(\TT)\right).
\]

Next, we want to show that $\ueps\rightarrow u$ in $C\left([0,T],L^2(\TT)\right)$. The method is classical
(see e.g., \cite{bertozzi-Majda}) and we only sketch the proof. We subtract the regularized problems corresponding to labels $\vartheta$ and $\varpi$:
\begin{multline*}
\pat \ueps-\pat \upi+\jeps\frac{\pax(\left(\jeps\ueps\right)^2)}{2}-\jpi\frac{\pax(\left(\jpi\upi\right)^2)}{2}\\
=\jeps\left(\Lambda^\gamma -\epsilon\Lambda^{1+\delta}\right)\jeps\ueps-\jpi\left(\Lambda^\gamma -\epsilon\Lambda^{1+\delta}\right)\jpi\upi.
\end{multline*}
From this equation, we obtain
$$
\|\ueps-\upi\|_{C\left([0,T],L^2(\TT)\right)}\leq C(T,u_0,\gamma,\epsilon,\delta)\max\{\varpi-\vartheta\},
$$
and we get that
$$
\ueps\rightarrow u\text{ in } C\left([0,T],L^2(\TT)\right).
$$
Using interpolation and the parabolic character of the equation, we have
$$
\ueps\rightarrow u\text{ in } C\left([0,T],H^\alpha(\TT)\right),
$$
which shows that $u$ is a classical solution. Uniqueness follows by energy estimates.

\textbf{Step 4: Regularized problems and compactness.} We define the regularized problems
\begin{equation}\label{KSreg2}
\pat \ueps+\pax\left(\frac{1}{2}\ueps^2\right)=\left(\Lambda^\gamma -\epsilon\Lambda^{1+\delta}\right)\ueps,%\text{ }x\in\TT, t>0.
\end{equation}
with initial data
$$
\ueps(0)= \jeps u_0.
$$
These problems have a global in time, smooth solution. Moreover, due to the energy estimates in the previous step, these solutions satisfy a uniform bound in the space
$$
\ueps\in L^p([0,T],H^\alpha(\RR)\cap L^\infty(\RR))
$$
for all $1\leq p\leq \infty$, and
$$
\ueps\in L^2([0,T],H^{\alpha+\frac{1+\delta}{2}}(\RR)).
$$
In particular, we get weak convergence in $L^2([0,T],H^{\alpha+\frac{1+\delta}{2}}(\RR))$ and weak-$*$ convergence in $L^\infty([0,T],L^\infty(\RR))$ of a subsequence to a function $u$. Moreover, by the weak lower semi-continuity of the norm, we have
$$
\|u\|_{L^2([0,T],H^{\alpha+\frac{1+\delta}{2}}(\RR))},
\ \|u\|_{L^\infty([0,T],L^\infty(\RR))}\leq C(\epsilon,\delta,\gamma,u_0).
$$

The dual space of $H^{(1+\delta)/2}(\TT)$ is $H^{-(1+\delta)/2}(\TT)$, and the corresponding norm of a function $f$ is given by
\begin{align*}
\|f\|_{H^{-(1+\delta)/2}(\TT)}&=\sup_{\|\psi\|_{H^{(1+\delta)/2}(\TT)}
\leq1} \left|\int_\TT f \psi dx\right|.
\end{align*}
We have
$$
H^{\alpha}(\TT)\hookrightarrow L^2(\TT)\hookrightarrow H^{-(1+\delta)/2}(\TT),
$$
where the first inclusion is compact and the second inclusion is continuous (see \cite{Valdinoci1}). To invoke the Aubin-Lions compactness Theorem (see Corollary 4, Section 8 in \cite{simon1986compact}) we need uniform bounds in the Bochner spaces
$$
\ueps\in L^\infty([0,T],H^{\alpha}(\TT)),
\qquad
\pat \ueps\in L^2([0,T],H^{-(1+\delta)/2}(\TT)).
$$
Multiplying \eqref{KSreg2} by $\psi\in H^{(1+\delta)/2}(\TT)$ and integrating by parts, we obtain
\begin{multline*}
\|\pat \ueps\|_{H^{-(1+\delta)/2}(\TT)}\leq \|\Lambda^{\frac{1-\delta}{2}}\ueps^2\|_{L^2(\TT)}+
\|\Lambda^{\frac{\gamma}{2}}\ueps\|_{L^2(\TT)}+\epsilon\|\Lambda^{\frac{1+\delta}{2}}\ueps\|_{L^2(\TT)}\\
\leq \|\Lambda^{\frac{1-\delta}{2}}\ueps\|_{L^2(\TT)}\|\ueps\|_{L^\infty(\TT)}+
\|\Lambda^{\frac{\gamma}{2}}\ueps\|_{L^2(\TT)}+\epsilon\|\Lambda^{\frac{1+\delta}{2}}\ueps\|_{L^2(\TT)}
\end{multline*}
Recalling that the energy estimates gives us uniform bounds
$$
\ueps\in L^2([0,T],H^{(1+\delta)/2}(\TT)) ,\text{ and }\ueps\in L^\infty([0,T],L^\infty(\TT)) ,
$$
and using Poincar\'e inequality, we get a uniform bound
$$
\pat \ueps\in L^2([0,T],H^{-(1+\delta)/2}(\TT)).
$$
Thus, we get

\begin{subequations}
\begin{align}
\ueps\rightharpoonup&  u\in L^2([0,T],H^{\alpha+\frac{1+\delta}{2}}(\TT)), \label{convergweak}\\
\pat \ueps\rightharpoonup& \pat u\in L^2([0,T],H^{-(1+\delta)/2}(\TT)).\label{convergweak2}
\end{align}
\end{subequations}
Applying the Aubin-Lions Lemma, we get that
\begin{subequations}
\begin{align}
\ueps\rightarrow& u\in C([0,T],L^2(\TT)),\label{convergLp}\\
\ueps\rightarrow& u\in C([0,T],L^p(\TT))\qquad  \mbox{for all $2\leq p<\infty$}.\label{convergLp2}
\end{align}
\end{subequations}
Then, using interpolation in Sobolev spaces, we get
\begin{equation}\label{converg}
\ueps\rightarrow u\in C([0,T],H^s(\TT)),\;0\leq s<\alpha.
\end{equation}

\textbf{Step 5: Convergence of the weak formulation.} We need to show that the limit $u$ of the regularized solutions in the previous step is a weak solution in the sense of Definition \ref{defi2b}. Let $\phi\in H^{(1+\delta)/2}(\TT)$ be a test function. Using the properties of mollifiers we obtain
$\ueps(0) \rightarrow u_0\text{ in }L^2$. To show convergence in the equation, we have to deal with the nonlinear term.

For $0<\delta<1/2$, we have $H^\delta\hookrightarrow L^{2/(1-2\delta)}$ and
\begin{align*}
&\int_{\TT}\Lambda^{(1+\delta)/2} \phi \Lambda^{(1-\delta)/2}\mathcal{H}\left(\ueps^2-u^2\right) \,dx\\
&\qquad\leq C(\phi)\|\Lambda^{(1-\delta)/2}\mathcal{H}\left(\ueps^2-u^2\right)\|_{L^2(\TT)}\\
&\qquad\leq C(\phi)\left(\|\Lambda^{(1-\delta)/2}(\ueps+u)\|_{L^{2/(1-2\delta)}(\TT)}\|\ueps-u\|_{L^{1/\delta}(\TT)}\right.\\ &\qquad\qquad\qquad\left.+\|\Lambda^{(1-\delta)/2}(\ueps-u)\|_{L^2(\TT)}\|\ueps+u\|_{L^\infty(\TT)}\right)\\
&\qquad\leq C(\phi)\left(\|\Lambda^{(1+\delta)/2}(\ueps+u)\|_{L^{2}(\TT)}\|\ueps-u\|_{L^{1/\delta}(\TT)}\right.\\ &\qquad\qquad\qquad\left.+\|\Lambda^{(1-\delta)/2}(\ueps-u)\|_{L^2(\TT)}\|\ueps+u\|_{L^\infty(\TT)}\right)\\
&\qquad\leq C(\phi,\epsilon,\delta,u_0,\gamma)\left(\|\ueps-u\|_{L^{1/\delta}(\TT)}\right.\\
&\qquad\qquad\qquad\left.+\|\Lambda^{(1-\delta)/2}(\ueps-u)\|_{L^2(\TT)}\right).
\end{align*}
For $\delta=1/2$, we use $H^{1/2}\hookrightarrow L^{4}$ to get
\begin{align*}
&\|\Lambda^{(1-\delta)/2}\mathcal{H}\left(\ueps^2-u^2\right)\|_{L^2(\TT)}\\
&\qquad\leq \left(\|\Lambda^{(1-\delta)/2}(\ueps+u)\|_{L^{4}(\TT)}\|\ueps-u\|_{L^{4}(\TT)}\right.\\ &\qquad\qquad\qquad\left.+\|\Lambda^{(1-\delta)/2}(\ueps-u)\|_{L^2(\TT)}\|\ueps+u\|_{L^\infty(\TT)}\right)\\
&\qquad\leq \left(\|\Lambda^{(1+\delta)/2}(\ueps+u)\|_{L^{2}(\TT)}\|\ueps-u\|_{L^{4}(\TT)}\right.\\
&\qquad\qquad\qquad \left.+\|\Lambda^{(1-\delta)/2}(\ueps-u)\|_{L^2(\TT)}\|\ueps+u\|_{L^\infty(\TT)}\right).
\end{align*}
For $1/2 < \delta \leq 1$, we have
$H^{\delta}\hookrightarrow L^{\infty}$
and
\begin{align*}
&\|\Lambda^{(1-\delta)/2}\mathcal{H}\left(\ueps^2-u^2\right)\|_{L^2(\TT)}\\
&\qquad\leq \left(\|\Lambda^{(1-\delta)/2}(\ueps+u)\|_{L^{\infty}(\TT)}\|\ueps-u\|_{L^{2}(\TT)}\right.\\ &\qquad\qquad\qquad\left.+\|\Lambda^{(1-\delta)/2}(\ueps-u)\|_{L^2(\TT)}\|\ueps+u\|_{L^\infty(\TT)}\right)\\
&\qquad\leq \left(\|\Lambda^{(1+\delta)/2}(\ueps+u)\|_{L^{2}(\TT)}\|\ueps-u\|_{L^{2}(\TT)}\right.\\ &\qquad\qquad\qquad\left.+\|\Lambda^{(1-\delta)/2}(\ueps-u)\|_{L^2(\TT)}\|\ueps+u\|_{L^\infty(\TT)}\right)
\end{align*}
Using \eqref{convergLp2} and \eqref{converg}, we obtain
\begin{equation}\label{convergnonlinear}
\sup_t\left|\int_{\TT}\Lambda^{(1+\delta)/2} \phi \Lambda^{(1-\delta)/2}\mathcal{H}\left(\ueps^2-u^2\right) dx\right|\rightarrow0.
\end{equation}
Next, we test against $\phi\in C^1([0,T],H^{(1+\delta)/2}(\TT))$ and integrate in time. Equation \eqref{convergweak} gives
$$
\int_0^T\int_\TT\Lambda^{s}\phi(t)\Lambda^{s}(\ueps(t)-u(t))\, dx dt\rightarrow 0\qquad 0\leq s\leq \alpha+\frac{1+\delta}{2},
$$
which ensures the convergence of the linear terms with $s=\gamma/2,(1+\delta)/2$, while \eqref{convergnonlinear} ensures the convergence of the nonlinear terms. Since
\[
C^1([0,T],H^{(1+\delta)/2}(\TT))
\]
is dense in
\[
L^2([0,T],H^{(1+\delta)/2}(\TT)),
\]
it follows that $u$ satisfies the weak formulation for every
\[
\phi\in L^2([0,T],H^{(1+\delta)/2}(\TT)).
\]
Taking $\phi$ independent of $t$, we find that the weak formulation holds almost everywhere in time, which completes the proof of the existence of weak solutions.

\textbf{Step 6: Uniqueness of weak solutions.} Suppose that $u_1$, $u_2$
are weak solutions of \eqref{KSivp} with the same initial data and let $w=u_1-u_2$. Testing against $w$, we have
\begin{eqnarray*}
\frac{1}{2}\frac{d}{dt}\|w\|_{L^2}^2&=&- \int_\RR w^2\pax u_1+w u_2\pax w dx\\
&&+\int_\RR w \Lambda^\gamma w-\epsilon\int_\RR w \Lambda^{1+\delta} w dx\\
&\leq& C\|w\|_{L^2}^2(\|u_1\|_{H^{1.5+\varepsilon}}+\|u_2\|_{H^{1.5+\varepsilon}}+1),
\end{eqnarray*}
and Gronwall's inequality implies that $w=0$.
\end{proof}

The proof of global existence for $\Omega=\RR$ is similar to the one for $\Omega = \TT$, but we need to modify
the proof of the $L^\infty$-estimate to account for the difference in the kernel of the fractional derivatives.

\begin{lem}\label{Linftyrecta}
If $u(x,t)$ is a smooth solution of \eqref{KSivp} on $\Omega = \RR$, then
$$
\|u(t)\|_{L^\infty(\RR)}\leq \|u_0\|_{L^\infty(\RR)}\exp\left(C(\epsilon,\gamma,\delta)t\right).
$$
\end{lem}
\begin{proof}
The fractional derivative on $\RR$ can be written as
$$
\Lambda^\alpha u(x)=\frac{c(\alpha)}{\pi}\pv \frac{u(x)-u(x-\eta)}{|\eta|^{1+\alpha}}d\eta
$$
where
$$
c(\alpha)=\Gamma(1+\alpha)\cos\left((1-\alpha)\frac{\pi}{2}\right).
$$
Let $x_t$ denote the point where $u$ reaches its maximum (this point is contained in a compact set in the real line since $u\in H^\alpha$ where $\alpha$ is certainly greater than $1/2$)
and assume that $\|u(t)\|_{L^\infty(\RR)}=u(x_t)$. Then, using Rademacher's Theorem
as before, we get
\begin{align*}
&\frac{d}{dt}\|u(t)\|_{L^\infty(\RR)}=\pat u(x_t)\\
&\qquad=\frac{1}{\pi}\pv \frac{(u(x_t)-u(x_t-\eta))\left(c(\gamma)|\eta|^{1+\delta-\gamma}-c(1+\delta)\epsilon\right)}{|\eta|^{2+\delta}}d\eta\\
&\qquad\leq \frac{1}{\pi}\int_{|\eta|>C(\epsilon,\gamma,\delta)} \frac{(u(x_t)-u(x_t-\eta))}{|\eta|^{1+\gamma}}d\eta\\
&\qquad\leq C(\epsilon,\gamma,\delta)\|u(t)\|_{L^\infty(\RR)}.
\end{align*}
Similarly, if $\|u(t)\|_{L^\infty(\RR)}=-u(x_t)$ where $x_t$ for the point where $u$ attains its minimum, we have
\begin{align*}
&\frac{d}{dt}\|u(t)\|_{L^\infty(\RR)}=-\pat u(x_t)\\
&\qquad=-\frac{1}{\pi}\pv \frac{(\|u(t)\|_{L^\infty(\RR)}+u(x_t-\eta))\left(c(1+\delta)\epsilon-c(\gamma)|\eta|^{1+\delta-\gamma}\right)}{|\eta|^{2+\delta}}\\
&\qquad\leq C(\epsilon,\gamma,\delta)\|u(t)\|_{L^\infty(\RR)},
\end{align*}
and it follows that
$$
\|u(t)\|_{L^\infty(\RR)}\leq \|u_0\|_{L^\infty(\RR)}\exp\left(C(\epsilon,\gamma,\delta) t\right).
$$
\end{proof}

Using Lemma~\ref{Linftyrecta} and the same ideas as in Theorem~\ref{teotoro}, we then get the following result.

\begin{teo}\label{teorecta}
Let $0 < \delta < 1$, $0 \le \gamma<1+\delta$, and $\epsilon > 0$.
If
\[
u_0\in H^\alpha(\RR)\cap L^\infty(\RR)
\]
with $\alpha\geq2+\delta$, then for every $0< T <\infty$ there exists a unique classical solution of \eqref{KSivp} such that
$$
u(x,t)\in C([0,T],H^\alpha(\RR)).
$$
Moreover, the solution gains regularity and satisfies
$$
u(x,t)\in L^2([0,T],H^{\alpha+\frac{1+\delta}{2}}(\RR)).
$$
\end{teo}

\subsection{The case $\delta=1$}
\label{sec:exist1}
In this case, equation \eqref{KS} becomes
\begin{equation}\label{KSdelta1}
\pat u+\pax\left(\frac{1}{2}u^2\right) =\Lambda^\gamma u +\epsilon\pax^2 u,\text{ }x\in\Omega, t>0,
\end{equation}
The previous proofs do not apply directly since they use a kernel representation
of $\Lambda^{1+\delta}$ which is not valid if $\delta=1$. Nevertheless, we have an analogous existence result.

\begin{teo}\label{teodelta1}
Let $u_0\in H^\alpha(\Omega)$ with $\alpha\geq1$ be the initial data for equation \eqref{KSdelta1}, where $\epsilon>0$, $0\le \gamma < 2$,
and $\Omega$ is $\TT$ or $\RR$. Then the following statements hold.
\begin{itemize}
\item If $\alpha\geq3$, then for every $0<T<\infty$ there exists a unique classical solution $$u(x,t)\in C([0,T],H^\alpha(\Omega)).$$

\item If $1\leq\alpha< 3$, then for every $0<T<\infty$ there exists a weak solution $$u(x,t)\in L^\infty([0,T],H^\alpha(\Omega))\cap C([0,T],L^2(\Omega)).$$
\item Moreover, the solution gains regularity and satisfies
$$
u(x,t)\in L^2([0,T],H^{\alpha+1}(\Omega)).
$$
\end{itemize}

\end{teo}
\begin{proof}
We give only the \emph{a priori} estimates. The proof then follows from the one for $0\le \delta < 1$ with minor changes.

The $L^2$ energy estimate is
$$
\frac{1}{2}\frac{d}{dt}\|u(t)\|_{L^2(\Omega)}^2+\frac{\epsilon}{2}\|\pax u\|_{L^2(\Omega)}^2=\|\Lambda^{\gamma/2} u\|_{L^2(\Omega)}^2-\frac{\epsilon}{2}\|\pax u\|_{L^2(\Omega)}^2.
$$
Using Fourier estimates and Gronwall's inequality, we obtain
$$
\|u(t)\|_{L^2(\Omega)}+\epsilon\int_0^t \|\pax u(s)\|_{L^2(\Omega)}^2ds\leq \|u_0\|_{L^2(\Omega)}^2\exp\left(c(\epsilon,\gamma)t\right).
$$
In particular
$$
\int_0^T \|u(s)\|_{L^\infty(\Omega)}^2ds\leq c\int_0^T \|\pax u(s)\|_{L^2(\Omega)}^2ds\leq C(T,u_0,\gamma,\epsilon).
$$
The $H^1$ energy estimate is
\begin{multline*}
\frac{1}{2}\frac{d}{dt}\|\pax u(t)\|_{L^2(\Omega)}^2\leq \frac{c}{\epsilon}\|u(t)\|_{L^\infty(\Omega)}^2\|\pax u(t)\|_{L^2(\Omega)}^2+\|\Lambda^{\gamma/2} u\|_{L^2(\Omega)}^2\\
-\frac{\epsilon}{2}\|\pax u\|_{L^2(\Omega)}^2\leq C(\epsilon,\gamma)(\|u(t)\|_{L^\infty(\Omega)}^2+1)\|\pax u(t)\|_{L^2(\Omega)}^2.
\end{multline*}
Also, using Sobolev and Gronwall inequalities we obtain
$$
\sup_{t\in[0,T]}\|u(t)\|_{L^\infty(\Omega)}^2\leq c\|\pax u(t)\|_{L^2(\Omega)}^2\leq C(\epsilon,\gamma,u_0,T).
$$
With these global estimate in $H^1$ and $L^\infty$, we can mimic the previous proof that used $H^\alpha$ norms.
\end{proof}

\section{Instant analyticity}
\label{sec:analyticity}
In this section, we prove that solutions of \eqref{KS} immediately gain some analyticity.
As in \cite{ccfgl} (see also \cite{AGM, CGO, AnalyticityKuramotoGrujic}), our proof is based on \emph{a priori} estimates in Hardy-Sobolev spaces for the complex extension of the function $u$ in a (growing) complex strip
$$
\BB_k(t) =\{x+i\xi: \mbox{$x\in\Omega$, $|\xi|<kt$}\},
$$
where $k$ is a positive constant. %which may  be chosen as in \eqref{defk} below.
We also consider a (shrinking) complex strip
$$
\VV_h(t) =\{x+i\xi:  \mbox{$x\in\Omega$, $|\xi|<h(t)$}\},
$$
where $h(t)$ is a positive, decreasing function.
When convenient, we do  not display the $t$-dependence of these strips explicitly.

We define the norms
\begin{align*}
\|u\|^2_{L^2(\BB_k)}&=\sum_{\pm}\int_\Omega |u(x\pm ikt)|^2dx,
\\
\|u\|_{H^n(\BB_k)}^2&=\|u\|_{L^2(\BB_k)}^2+\|\pax^n u\|^2_{L^2(\BB_k)},
\end{align*}
with their analogous counterparts for the strip $\VV_h$.
The corresponding function spaces have the same flavour as the Gevrey classes used in \cite{FerrariTiti, FTkuramoto2}. In particular, the tools in \cite{FerrariTiti} may be adapted to get $u(x,t)\in G^1_t(\Omega)$, which implies the analyticity for real spatial arguments $x$.

\begin{teo}\label{SE}Let $u$ be a classical solution of \eqref{KSivp} with (real-valued) initial data $u_0$, where $\epsilon >0$ and $\gamma$, $\delta$ satisfy \eqref{gammadelta}. Then the following statements hold.
\begin{itemize}
\item If $u_0\in H^3(\Omega)$ and $k > 0$, then there exists a time $T(k, u_0,\epsilon,\delta,\gamma) > 0$ such that $u$ continues analytically into the strip $\BB_k(t)$ for $0<t<T(k, u_0,\epsilon,\delta,\gamma)$.
\item If $u_0\in H^3(\Omega)$ continues to an analytic function in a complex strip of width $h_0 > 0$, then there exists a time $T(u_0,\epsilon,\delta,\gamma)$ and a positive decreasing function $h : [0,T) \to (0,\infty)$ such that $h(0) = h_0$ and  $u$ continues analytically into the strip $\VV_h(t)$ for $0<t<T(u_0,\epsilon,\delta,\gamma)$ with finite $H^3(\VV_h)$-norm. %The function $h$ may be chosen as in \eqref{defh} below.
\end{itemize}
\end{teo}
\begin{proof}
\textbf{Step 1: Growing strip.} We prove the result in the case $\Omega=\TT$; the case $\Omega=\RR$ is similar. We write $z=x\pm ikt$. Then the extended equation is
\begin{equation}\label{KScomplex}
\pat u(z,t)+u(z,t)\pax u(z,t)=\left(\Lambda^\gamma -\epsilon\Lambda^{1+\delta}\right)u(z,t),\text{ }x\in\Omega, t>0.
\end{equation}

First, we study the evolution of $\|u\|_{H^3(\BB_k)}$. Since we consider periodic solutions with zero mean, it follows from Poincar\'e inequalities
that we only need to estimate the $L^2$ norm of the third derivative.

Using Plancherel's theorem, we have
$$
\frac{d}{dt}\|\pax^3 u\|^2_{L^2(\BB_k)}=2\Re\int_\TT \pax^3\bar{u}(z)\left(\pat \pax^3 u(z)\pm ik\pax^4 u(z)\right)dx,
$$
and from \eqref{KScomplex}, we get that
\begin{equation}
\pat\pax^3u=-3(\pax^2u)^2-4\pax u\pax^3u-u\pax^4u+\Lambda^\gamma\pax^3u-\Lambda^{1+\delta}\pax^3u.
\label{d3eq}
\end{equation}
We have the following estimates:
\begin{align*}
A_1&=-3\int_\TT (\pax^2u(z))^2\pax^3\bar{u}(z)dx\leq C\|\pax^3 u\|_{L^2(\BB_k)}\|\pax^2 u\|_{L^2(\BB_k)}\|\pax^2 u\|_{L^\infty(\BB_k)}\\
&\leq C\|\pax^3 u\|_{L^2(\BB_k)}^3,
%\end{align*}
\\
%\begin{align*}
A_2&=-4\int_\TT \pax u(z)|\pax^3u(z)|^2dx\leq C\|\pax^3 u\|^2_{L^2(\BB_k)}\|\pax u\|_{L^\infty(\BB_k)}\\
&\leq C\|\pax^3 u\|_{L^2(\BB_k)}^3,
%\end{align*}
\\
%\begin{align*}
A_3&=\pm ik\int_\TT \pax^3 \bar{u}(z)\pax^4 u(z)dx\\
&=\mp ik\int_\TT \pax^3 \bar{u}(z)\Lambda \mathcal{H} \pax^3 u(z)dx\\
&=\mp ik\int_\TT \Lambda^{1/2}\pax^3 \bar{u}(z)\Lambda^{1/2} \mathcal{H} \pax^3 u(z)dx\leq 2k\|\Lambda^{1/2}\pax^3 u\|_{L^2(\BB_k)}^2.
\end{align*}
Moreover, we have
\begin{align*}
A_4&=\Re \int_\TT \pax^3 \bar{u}(z)u(z)\pax^4 u(z)dx\\
&=\int_\TT\Re\pax^3 u\Re\pax^4 u\Re u+\Im\pax^3 u\Im\pax^4 u\Re u dx\\
&\qquad+\int_\TT-\Re\pax^3 u\Im\pax^4 u\Im u+\Re\pax^4 u\Im\pax^3 u\Im u dx\\
&=-\frac{1}{2}\int_\TT|\pax^3 u|^2\Re \pax udx\\
&\qquad-2\int_\TT\Re\pax^3 u\Im\pax^4 u\Im u+\Re\pax^3 u\Im\pax^3 u\Im \pax u dx\\
&=-\frac{1}{2}\int_\TT|\pax^3 u|^2\Re \pax udx+\int_\TT\Re\pax^3 u\Im\pax^3 u\Im \pax u dx\\
&\qquad-2\int_\TT\left[\Lambda^{1/2},\Im u\right]\Re\pax^3 u\Lambda^{1/2}\mathcal{H}\Im\pax^3 u dx\\
&\qquad-2\int_\TT\Im u\Lambda^{1/2}\Re\pax^3 u\Lambda^{1/2}\mathcal{H}\Im\pax^3 u dx,
\end{align*}
so, using the commutator estimate (see Lemma~\ref{commutator})
$$
\left\|\left[\Lambda^{1/2},F \right]G\right\|_{L^2}\leq c\|\pax F\|_{L^\infty}\|G\|_{L^2},
$$
we get that
\begin{align*}
A_4&\leq C\|\pax^3 u\|_{L^2(\BB_k)}^2\|\pax u\|_{L^\infty(\BB_k)}\\
&\qquad+C\|\pax^3 u\|_{L^2(\BB_k)}\|\pax u\|_{L^\infty(\BB_k)}\|\Lambda^{1/2}\Im\pax^3 u\|_{L^2(\BB_k)}\\
&\qquad+2\|\Lambda^{1/2}\pax^3 u\|^2_{L^2(\BB_k)}\|\Im u\|_{L^\infty(\BB_k)}\\
&\leq C(\|\pax^3 u\|^4_{L^2(\BB_k)}+1)\\
&\qquad+\|\Lambda^{1/2}\pax^3 u\|^2_{L^2(\BB_k)}\left(2\|\Im u\|_{L^\infty(\BB_k)}+1\right).
\end{align*}
Let $\lambda>\|u_0\|_{L^\infty}$ be a positive constant. Putting these results together and using Poincar\'e's inequality, we get
\begin{align*}
\frac{d}{dt}\|\pax^3 u\|^2_{L^2(\BB_k)}\leq& C(\|\pax^3 u\|^4_{L^2(\BB_k)}+1)
\\&
+\|\Lambda^{1/2}\pax^3 u\|^2_{L^2(\BB_k)}\left(2\|\Im u\|_{L^\infty(\BB_k)}-2\lambda+2\lambda+2k+1\right)
\\&
+\|\Lambda^{\gamma/2}\pax^3 u\|^2_{L^2(\BB_k)}-\epsilon\|\Lambda^{(1+\delta)/2}\pax^3 u\|^2_{L^2(\BB_k)}\\
\leq& C(\|\pax^3 u\|_{L^2(\BB_k)}+1)^4
%\\&
+\|\Lambda^{1/2}\pax^3 u\|^2_{L^2(\BB_k)}\left(2\|\Im u\|_{L^\infty(\BB_k)}-2\lambda\right)
\\&
+2\left(\lambda+k+1\right)\|\Lambda^{\max\{1,\gamma\}/2}\pax^3 u\|^2_{L^2(\BB_k)}
%\\&
-\epsilon\|\Lambda^{(1+\delta)/2}\pax^3 u\|^2_{L^2(\BB_k)}.
\end{align*}
Define a constant $C(\lambda,k,\epsilon,\delta,\gamma) > 0$ by
\begin{align}
\begin{split}
C(\lambda,k,\epsilon,\delta,\gamma) &= \max_{\xi\in \RR}\left[
2\left(\lambda+k+1\right)|\xi|^{\max\{1,\gamma\}}-\epsilon|\xi|^{1+\delta}\right]
\\
&=2\left(\lambda+k+1\right)\left(\frac{\max\{1,\gamma\}2\left(\lambda+k+1\right)}{\epsilon(1+\delta)}\right)^{\frac{\max\{1,\gamma\}}{1+\delta-\max\{1,\gamma\}}}\\
&\quad-\epsilon\left(\frac{\max\{1,\gamma\}2\left(\lambda+k+1\right)}{\epsilon(1+\delta)}\right)^{\frac{1+\delta}{1+\delta-\max\{1,\gamma\}}}.
\end{split}
\label{defC_analytic}
\end{align}

Then, using Plancherel's theorem, we get that
\begin{align*}
&2\left(\lambda+k+1\right)\|\Lambda^{\max\{1,\gamma\}/2}\pax^3 u\|^2_{L^2(\BB_k)}
-\epsilon\|\Lambda^{(1+\delta)/2}\pax^3 u\|^2_{L^2(\BB_k)}
\\
&\qquad\qquad \le C(\lambda,k,\epsilon,\delta,\gamma)\|\pax^3 u\|^2_{L^2(\BB_k)},
\end{align*}
and therefore
\begin{align*}
\frac{d}{dt}\|\pax^3 u\|^2_{L^2(\BB_k)}
\leq& C(\|\pax^3 u\|_{L^2(\BB_k)}+1)^4
%\\&
+ 2 \|\Lambda^{1/2}\pax^3 u\|^2_{L^2(\BB_k)}\left(\|\Im u\|_{L^\infty(\BB_k)}-\lambda\right)\\
&+C(\lambda,k,\epsilon,\delta,\gamma)\|\pax^3 u\|^2_{L^2(\BB_k)}.
\end{align*}

We define a new energy by
$$
\|u\|_{\BB_k}=\|\pax^3 u\|^2_{L^2(\BB_k)}+\|d^\lambda[u]\|_{L^\infty(\BB_k)}
$$
where
$$
d^\lambda[u](z)=\frac{1}{\lambda^2-|u(z)|^2}.
$$

Note that $|u(z)|<\lambda$ as long as $\|u\|_{\BB_k}$ remains finite.
We need a bound for the remaining term in the energy $\|u\|_{\BB_k}$. Using \eqref{KScomplex} and Sobolev embedding to estimate $\pat u$, we have
$$
\frac{d}{dt}d^\lambda[u]\leq 4 d^\lambda[u]^2\|u\|_{L^\infty(\BB_k)}\|\pat u\|_{L^\infty(\BB_k)}\leq C(\|u\|_{\BB_k}+1)^3d^\lambda[u]
$$
Thus, we obtain
$$
d^\lambda[u](t+h)\leq d^\lambda[u](t)\exp
\left(\int_t^{t+h} C(\|u\|_{\BB_k}+1)^3 ds\right).
$$
Finally, we have
\begin{align*}
\frac{d}{dt}\|d^\lambda[u]\|_{L^\infty(\TT)}&=\lim_{h\rightarrow0}\frac{\|d^\lambda[u](t+h)\|_{L^\infty(\TT)}-\|d^\lambda[u](t)\|_{L^\infty(\TT)}}{h}\\
&\leq C(\|u\|_{\BB_k}+1)^4.
\end{align*}
It follows that
\begin{align*}
\frac{d}{dt}\|u\|_{\BB_k}&=\frac{d}{dt}\|\pax^3 u\|^2_{L^2(\BB_k)}+\frac{d}{dt}\|d^\lambda[u]\|_{L^\infty(\TT)}\\
&\leq c(\|\pax^3 u\|_{L^2(\BB_k)}+1)^4+C(\lambda,k,\epsilon,\delta,\gamma)\|\pax^3 u\|^2_{L^2(\BB_k)}+c(\|u\|_{\BB_k}+1)^4\\
&\leq c(\|u\|_{\BB_k}+1)^4+C(\lambda,k,\epsilon,\delta,\gamma)\|u\|_{\BB_k}.
\end{align*}
Thus,
\begin{align*}
\|u(t)\|_{\BB_k}&\leq \frac{\sqrt[3]{C(\lambda,k,\epsilon,\delta,\gamma)}\exp\left(C(\lambda,k,\epsilon,\delta,\gamma)
\left[\frac{\log\left(\frac{\|u(0)\|_{\BB_k}}{\sqrt[3]{c\|u(0)\|_{\BB_k}^3
+C(\lambda,k,\epsilon,\delta,\gamma)}}\right)}{C(\lambda,k,\epsilon,\delta,\gamma)}+t\right]
\right)}{\sqrt[3]{1-c\exp\left(3C(\lambda,k,\epsilon,\delta,\gamma)
\left[\frac{\log\left(\frac{\|u(0)\|_{\BB_k}}{\sqrt[3]{c\|u(0)\|_{\BB_k}^3
+C(\lambda,k,\epsilon,\delta,\gamma)}}\right)}{C(\lambda,k,\epsilon,\delta,\gamma)}+t\right]\right)}}.
\end{align*}
The time of existence of analytic solutions is then at least
\begin{align}\label{time}
\begin{split}
T(k, u_0, \epsilon,\delta,\gamma)&=\frac{\log\left(\frac{C(\lambda,k,\epsilon,\delta,\gamma)}{\left(\|\pax^3 u_0\|^2_{L^2}+\frac{1}{\lambda^2-\|u_0\|_{L^\infty}^2}\right)^3c}+1\right)}{3C(\lambda,k,\epsilon,\delta,\gamma)},
\end{split}
\end{align}
where $C(\lambda,k,\epsilon,\delta,\gamma)$ is given by \eqref{defC_analytic}, and we may choose
$\lambda = \sqrt{2}\|u_0\|_\infty$, for example.

Now we approximate this problem using an analytic mollifier such as the heat kernel. The regularized problems have entire solutions and satisfy the same \emph{a priori} bounds. Using the uniqueness of classical solutions, we obtain the first part of the result.

\textbf{Step 2: Shrinking strip} As before, we consider the evolution in the Hardy-Sobolev spaces in the strip $\VV_h$. We write $z=x\pm ih(t)$. Notice that since the solution is real for real $z$ we have
\[
\pax^k u(x\pm ih(t))-\pax^k u(x\pm i 0)=\int_\Gamma \pax^{k+1} u(x\pm \zeta)d\zeta=\int_0^{h(t)} i\pax^{k+1} u(x\pm i\theta)d\theta.
\]
Thus, using the Hadamard Three Lines Theorem, we get
\begin{align*}
\left|\pax^k u(x\pm ih(t))-\pax^k u(x\pm i 0)\right|
&\leq h(t)\sup_{x\in\TT}\sup_{|\theta|< h(t)}|\pax^{k+1} u(x\pm i\theta)|
\\
&\leq h(t)\|\pax^{k+1} u\|_{L^\infty(\VV_h)}.
\end{align*}
Using Lemma \ref{complexevolution} and equation \eqref{d3eq} for $\pat\pax^3u$, we have
\begin{align*}
\frac{d}{dt}\|\pax^3(t)\|^2_{L^2(\VV_h)}\leq& \frac{h'(t)}{10}\sum_{\pm}\int_\TT\Lambda\pax^3u(z)\overline{\pax^3u(z)}dx\\
&-10h'(t)\sum_{\pm}\int_\TT\Lambda\pax^3u(x)\overline{\pax^3u(x)}dx\\
&+2\Re\sum_{\pm}\int_\TT\pat\pax^3u(z)\overline{\pax^3u(z)}dx
\\
=&J_1+J_2+J_3+J_4,
\end{align*}
where
\begin{align*}
J_1&=\frac{h'(t)}{10}\sum_{\pm}\int_\TT\Lambda \pax^3 u(z)\overline{\pax^3 u(z)}dx,
\\
J_2&=-10h'(t)\sum_{\pm}\int_\TT\Lambda \pax^3 u(x)\overline{\pax^3 u(x)}dx\\
J_3&=2\Re\int_\TT\left[-3(\pax^2u)^2-4\pax u\pax^3u-u\pax^4u\right]\overline{\pax^3 u(z)}dx
\\
&=K_1+K_2+K_3,
\\
J_4&=2\Re \sum_{\pm}\int_\TT\left(\Lambda^\gamma-\epsilon\Lambda^{1+\delta} \right)\pax^3 u(x)\overline{\pax^3 u(x)}dx.
\end{align*}
We have the estimates
\begin{align*}
J_2&\leq 20|h'(t)|\|u_0\|^2_{H^{3.5}}\exp\left(\exp\left(C(\epsilon,\delta,\gamma,\|u_0\|_{L^\infty}(1+t))\right)\right),
\\
J_4&\leq 2\left(\frac{\gamma}{\epsilon(1+\delta)}\right)^{\frac{1}{1+\delta-\gamma}}\|\pax^3 u\|^2_{L^2(\VV_h)}.
\end{align*}
Moreover, following the previous ideas, and using Gagliardo-Nirenberg and Sobolev inequalities, we find that
$$
K_1+K_2\leq C\|\pax^3 u\|_{L^2(\VV_h)}^2\|\pax u\|_{L^\infty(\VV_h)}\leq C\|\pax^3 u\|^3_{L^2(\VV_h)}.
$$
We also have
\begin{align*}
K_3=&2\int_\TT\Re u\Re\pax^4 u\Re\pax^3 u +\Re u\Im\pax^4 u\Im\pax^3 u dx\\
&+2\int_\TT -\Im u\Re\pax^4 u\Im\pax^3 u + \Im u\Im\pax^4 u\Re\pax^3 udx\\
\leq& C\|\pax^3 u\|_{L^2(\VV_h)}^2\|\pax u\|_{L^\infty(\VV_h)}-4\int_\TT \Im u\Re\pax^4 u\Im\pax^3 u dx\\
\leq& C\|\pax^3 u\|^3_{L^2(\VV_h)}-4\int_\TT \Lambda^{1/2}\mathcal{H}\Re\pax^3 u \Lambda^{1/2}\left(\Im u\Im\pax^3 u \right)dx.
\end{align*}
The last integral can be written in terms of a commutator as
$$
\int_\TT \Lambda^{1/2}\mathcal{H}\Re\pax^3 u \left[\Lambda^{1/2},\Im u\right]\Im\pax^3 u dx+\int_\TT \Lambda^{1/2}\mathcal{H}\Re\pax^3 u \Im u \Lambda^{1/2}\Im\pax^3 u\, dx,
$$
and using Lemma \ref{commutator}, we get
\begin{eqnarray*}
K_3&\leq& C\|\pax^3 u\|^3_{L^2(\VV_h)}+C\|\Lambda^{1/2}\pax^3 u\|_{L^2(\VV_h)}\|\pax \Im u\|_{L^\infty(\VV_h)}\|\pax^3 u\|_{L^2(\VV_h)}\\
&&-4\int_\TT \Lambda^{1/2}\mathcal{H}\Re\pax^3 u \Im u \Lambda^{1/2}\Im\pax^3 u dx\\
&\leq& C\left(\|\pax^3 u\|_{L^2(\VV_h)}+1\right)^3\\
&&+C\left(\|\pax\Im u\|^2_{L^\infty(\VV_h)}+\|\Im u\|_{L^\infty(\VV_h)}\right)\|\Lambda^{1/2}\pax^3 u\|^2_{L^2(\VV_h)}\\
&\leq& C\left(\|\pax^3 u\|_{L^2(\VV_h)}+1\right)^3\\
&&+Ch(t)\left(\|\pax^3 u\|_{L^2(\VV_h)}+1\right)^2\|\Lambda^{1/2}\pax^3 u\|^2_{L^2(\VV_h)}.
\end{eqnarray*}
Collecting the bounds for $K_3$ and for $J_1$, we have
\begin{eqnarray*}
K_3+J_1&\leq& C\left(\|\pax^3 u\|_{L^2(\VV_h)}+1\right)^3\\
&&+\left(Ch(t)\left(\|\pax^3 u\|_{L^2(\VV_h)}+1\right)^2+10h'(t)\right)\|\Lambda^{1/2}\pax^3 u\|^2_{L^2(\VV_h)},
\end{eqnarray*}
and, choosing
\begin{equation}
\label{defh}
h(t)=  h(0)\exp\left(-10C\int_0^t \left(\|\pax^3 u(s)\|_{L^2(\VV_h)}+1\right)^2ds \right),
\end{equation}
we obtain
\begin{multline*}
\frac{d}{dt}\|\pax^3 u(t)\|^2_{L^2(\VV_h)}\leq C\left(\|\pax^3 u\|_{L^2(\VV_h)}+1\right)^3\\
+C\left(\|\pax^3 u\|_{L^2(\VV_h)}+1\right)^2\|u_0\|^2_{H^{3.5}}\exp\left(\exp\left(C(\epsilon,\delta,\gamma,\|u_0\|_{L^\infty}(1+t))\right)\right).
\end{multline*}
Finally, we use a standard Galerkin approximation method to obtain a local solution that satisfies these estimates, which completes the proof.
\end{proof}

In the previous proof, we can choose the parameter $k > 0$ that determines the strips of analyticity in any way we wish, but we get shorter existence times for larger values of $k$, so we cannot conclude that the solution is entire for $t > 0$.

To obtain an explicit estimate for the width of a strip that depends only on the initial data (and the parameters in the equation), we choose
\begin{equation}
k=\left(\|\pax^3 u_0\|^2_{L^2}+\frac{1}{\lambda^2-\|u_0\|_{L^\infty}^2}\right)^3,
\qquad
\lambda = \sqrt{2}\|u_0\|_\infty
\label{defk}
\end{equation}
in the proof of Theorem~\ref{SE}. Then the corresponding time $T$ of analyticity is given by \eqref{time}, and the width of the strip of analyticity at time $T$ is at least $kT$.
%Writing
%\begin{multline}\label{Ebonita}
%\mathcal{E}=\frac{2\left(\lambda+k+1\right)
%\left(\frac{\max\{1,\gamma\}2\left(\lambda+k+1\right)}{\epsilon(1+\delta)}\right)^{\frac{\max\{1,\gamma\}}{1+\delta-\max\{1,\gamma\}}}}{\left(\|\pax^3 %u_0\|^2_{L^2}+\frac{1}{\lambda^2-\|u_0\|_{L^\infty}^2}\right)^3}\\
%-\frac{\epsilon\left(\frac{\max\{1,\gamma\}2\left(\lambda+k+1\right)}{\epsilon(1+\delta)}\right)^{\frac{1+\delta}{1+\delta-\max\{1,\gamma\}}}}{\left(\|\pax^3 %u_0\|^2_{L^2}+\frac{1}{\lambda^2-\|u_0\|_{L^\infty}^2}\right)^3},
%\end{multline}
Using the preceding equations, we find that
\begin{equation}\label{width}
%kT=O\left(\frac{\log\left(\frac{\mathcal{E}}{c}+1\right)}{3c\frac{\mathcal{E}}{c}}\right).
kT= \frac{\log\left(\mathcal{E}/{c}+1\right)}{3 \mathcal{E}},
\end{equation}
where $c$ is a constant, %depending on what?
and $\mathcal{E}$ is given by
\begin{align}
\begin{split}
\mathcal{E}& =\frac{2\left(\sqrt{2}\|u_0\|_{L^\infty(\TT)}+k+1\right)\left(\frac{\max\{1,\gamma\}2\left(\sqrt{2}\|u_0\|_{L^\infty(\TT)}+k+1\right)}{\epsilon(1+\delta)}\right)^{\frac{\max\{1,\gamma\}}{1+\delta-\max\{1,\gamma\}}}}{\left(\|\pax^3 u_0\|^2_{L^2}+\frac{1}{\|u_0\|_{L^\infty}^2}\right)^3}
\\
&\qquad\qquad -\frac{\epsilon\left(\frac{\max\{1,\gamma\}2\left(\sqrt{2}\|u_0\|_{L^\infty(\TT)}+k+1\right)}{\epsilon(1+\delta)}\right)^{\frac{1+\delta}{1+\delta-\max\{1,\gamma\}}}}{\left(\|\pax^3 u_0\|^2_{L^2}+\frac{1}{\|u_0\|_{L^\infty}^2}\right)^3}.
\end{split}
\label{defE}
\end{align}

Finally, we remark that by using this smoothing effect, one can prove the ill-posedness in Sobolev spaces of the evolution problem backward in time.

\begin{corol}
There are solutions $\tilde{u}$ to the backward in time equation \eqref{KS}, such that $\|\tilde{u}\|_{H^4}(0)<\epsilon$ and $\|\tilde{u}\|_{H^4}(\mu)=\infty$ for all $\epsilon > 0$ and sufficiently small $\mu > 0$.
\end{corol}
\begin{proof}
The proof follows the idea in \cite{AGM,CGO}. We consider the solution (forward in time) $u^{\nu}$ to the equation \eqref{KS} with initial data $u(x,0)=\nu v(x)$ where $v\in H^3$, $v\notin H^4$ $0<\nu<1$. Now define $\tilde{u}^ {\nu,\mu}(x,t)=u^ \nu(x,-t+\mu)$ for fixed, small enough $0<\mu(v)\ll 1$. This function is analytic at time $0$ but it does not belong to $H^4$ at time $\mu$. Taking $0<\nu\ll 1$ we conclude the proof.
\end{proof}

\section{Large time dynamics}
\label{sec:absorbing}
%\subsection{Absorbing set in $L^p$}
In this section we prove the existence of an absorbing ball in $L^p$ for the problem \eqref{KS} in the periodic case $\Omega=\TT$. We will require a Lemma similar to the results in \cite{goodmankuramoto, NSTkuramoto,temambook}:

\begin{lem}\label{aux1} Let $M\in \NN$, $\delta > 0$, and $x_0\in \TT$. Then there exists a smooth, periodic function $b^{x_0}_M\in C^\infty(\TT)$ and a constant
\[
C_1(\delta,M) = c_1(\delta)\left(\frac{1}{M^{1+\delta}}+\frac{1}{\delta M^{\delta}}\right)^{1/2}
\]
such that the following inequality holds: for every $u\in C^\infty(\TT)$ with  $u(x_0)=0$,
$$
\left|\int_\TT b^{x_0}_M(x)u^2(x,t)dx\right|\leq C_1(\delta,M)\|\Lambda^{\frac{1+\delta}{2}}u\|_{L^2(\TT)}^2.
$$
\end{lem}
\begin{proof}
We define
$$
b^{x_0}_M(x)=\sum_{|\xi|\leq M}e^{-i\xi (x-x_0)}.
$$
We have
\begin{align*}
\int_\TT b^{x_0}_M(x)u^2(x,t)dx &=\sum_{|\xi|\leq M}\int_\TT u^2(x,t)e^{-i\xi (x-x_0)}dx\\
&=\sum_{|\xi|\leq M}\int_\TT u^2(x+x_0,t)e^{-i\xi x}dx
\\
&= 2\pi\sum_{|\xi|\leq M}\widehat{g}(\xi),
\end{align*}
where $g(x)=u^2(x+x_0)$. Since $\sum\widehat{g}(\xi)=g(0)$, it follows from the definition of $x_0$ that $\sum\widehat{g}(\xi)= 0$, and therefore
\begin{align*}
\left|\sum_{|\xi|\leq M}\widehat{g}(\xi)\right|&\leq \left|\sum_{|\xi| > M}\widehat{g}(\xi)\right|
\\
%&\leq \left|\sum_{|\xi|\geq M}\frac{|\xi|^{\frac{1+\delta}{2}}\widehat{g}(\xi)}{|\xi|^{\frac{1+\delta}{2}}}\right|\\
&\leq \left(\sum_{|\xi| >  M}|\xi|^{1+\delta}\left(\widehat{g}(\xi)\right)^2\right)^{1/2}\left(\sum_{|\xi| > M}\frac{1}{|\xi|^{1+\delta}}\right)^{1/2}\\
&\leq \frac{1}{\sqrt{2\pi}}\|\Lambda^{\frac{1+\delta}{2}}g\|_{L^2(\TT)}\left(\frac{1}{M^{1+\delta}}+\frac{1}{\delta M^{\delta}}\right)^{1/2}.
\end{align*}
The Kato-Ponce inequality then implies that there is a constant $c_1(\delta)$ such that
$$
\left|\int_\TT b^{x_0}_M(x)u^2(x,t)dx\right|\leq c_1(\delta)\|\Lambda^{\frac{1+\delta}{2}}u\|_{L^2(\TT)}^2\left(\frac{1}{M^{1+\delta}}+\frac{1}{\delta M^{\delta}}\right)^{1/2},
$$
which proves the result.
\end{proof}

Next, we prove that solutions of \eqref{KS} remain uniformly bounded in $L^p$. The key step is to prove the existence of an absorbing set in $L^2$, and we do this following the ideas of \cite{goodmankuramoto, NSTkuramoto}.

\begin{teo}\label{lowerset}
Suppose that $u_0\in H^\alpha(\TT)$, where $\alpha>1$, has zero mean. Then the solution $u$ of the initial-value problem \eqref{KSivp} in the periodic case satisfies
$$
\limsup_{t\rightarrow\infty} \|u(t)\|_{L^2(\TT)}\leq r_2(\epsilon,\delta,\gamma),
$$
$$
\|u(t)\|_{L^2(\TT)}\leq \max\{\|u_0\|_{L^2(\TT)},r_2\}=R(\epsilon,\delta,\gamma).
$$
Moreover, for $2<p\leq\infty$ and $0<\delta<1$, we have
$$
\limsup_{t\rightarrow\infty}\|u(t)\|_{L^p(\TT)}\leq r_2^{2/p}
\left(\max\left\{\sqrt{\frac{3}{\pi}}R,C\left(\delta\right)R\right\}\right)^{1-2/p}.
$$
\end{teo}
\begin{proof}
We start by assuming that the initial data is odd.

\textbf{Step 1: Absorbing set in $L^2$} Let $s$ be a smooth, periodic function, which we will choose later. We compute that
\begin{align*}
\frac{1}{2}\frac{d}{dt}\|u(t)-s\|_{L^2(\TT)}^2&= \|\Lambda^{\gamma/2}u\|_{L^2(\TT)}^2-\epsilon\|\Lambda^{(1+\delta)/2}u\|_{L^2(\TT)}^2-\int_\TT\pax s\frac{u^2}{2}dx\\
&-\int_\TT\Lambda^{(1+\delta)/2}u\left(\epsilon
\Lambda^{(1+\delta)/2}s+\Lambda^{\gamma-(1+\delta)/2}s\right)dx.
\end{align*}
Using the inequality
$$
2|\xi|^\gamma\leq \frac{\epsilon}{3}|\xi|^{1+\delta}+\left(\frac{6\gamma}{(1+\delta)\epsilon}\right)^{\frac{1}{1+\delta-\gamma}},\text{ for all }\xi\in\RR
$$
and the Plancherel theorem, we get
\begin{align*}
\frac{1}{2}\frac{d}{dt}\|u(t)-s\|_{L^2(\TT)}^2&\leq-\|\Lambda^{\gamma/2}u\|_{L^2(\TT)}^2-\frac{2\epsilon}{3}\|\Lambda^{(1+\delta)/2}u\|_{L^2(\TT)}^2-\|u\|_{L^2(\TT)}^2\\
&\quad+\int_\TT\left(\lambda-\frac{\pax s}{2}\right)u^2dx
\\&\quad
+\int_\TT\Lambda^{(1+\delta)/2}u\left(-\epsilon
\Lambda^{(1+\delta)/2}s+\Lambda^{\gamma-(1+\delta)/2}s\right)dx,
\end{align*}
where
\begin{equation}\label{lambdadefi}
\lambda=\left(\frac{6\gamma}{(1+\delta)\epsilon}\right)^{\frac{1}{1+\delta-\gamma}}+1.
\end{equation}
Then, using the Young and Cauchy-Schwarz inequalities, we obtain
\begin{align*}
\frac{1}{2}\frac{d}{dt}\|u(t)-s\|_{L^2(\TT)}^2&\leq-\|\Lambda^{\gamma/2}u\|_{L^2(\TT)}^2-\|u\|_{L^2(\TT)}^2-\frac{\epsilon}{3}\|\Lambda^{(1+\delta)/2}u\|_{L^2(\TT)}^2\\
&\quad+\int_\TT\left(\lambda-\frac{\pax s}{2}\right)u^2dx\\
&\quad+\frac{3}{\epsilon}\int_\TT\left(\left(-\epsilon
\Lambda^{(1+\delta)/2}+\Lambda^{\gamma-(1+\delta)/2}\right)s\right)^2dx.
\end{align*}

Since the odd symmetry is preserved by \eqref{KS} and $u_0$ is odd, we have $u(0,t)=0$.
For $M\in \NN$,
we choose $s$ such that
\begin{equation}\label{sdefi}
\pax s(x)= - 2\lambda \sum_{0<|\xi|\leq M}e^{-i\xi x} = -2\lambda \left[b^{0}_M(x) - 1\right].
\end{equation}
Then from the preceding inequality and Lemma \ref{aux1}, we get
\begin{align*}
&\frac{1}{2}\frac{d}{dt}\|u(t)-s\|_{L^2(\TT)}^2\\
&\qquad\leq-\|\Lambda^{\gamma/2}u\|_{L^2(\TT)}^2-\|u\|_{L^2(\TT)}^2-\frac{\epsilon}{3}\|\Lambda^{(1+\delta)/2}u\|_{L^2(\TT)}^2\\
&\qquad\quad+\int_\TT b^{\lambda,0}_M u^2dx+\frac{3}{\epsilon}\int_\TT\left(\left(-\epsilon
\Lambda^{(1+\delta)/2}+\Lambda^{\gamma-(1+\delta)/2}\right)s\right)^2dx\\
&\qquad\leq-\|\Lambda^{\gamma/2}u\|_{L^2(\TT)}^2-\|u\|_{L^2(\TT)}^2-\frac{\epsilon}{3}\|\Lambda^{(1+\delta)/2}u\|_{L^2(\TT)}^2\\
&\qquad\quad+c_1\lambda\|\Lambda^{(1+\delta)/2}u\|_{L^2(\TT)}^2\left(\frac{1}{M^{1+\delta}}+\frac{1}{\delta M^{\delta}}\right)^{1/2}+\frac{6}{\epsilon}\|\Lambda s\|_{L^2(\TT)}^2.
\end{align*}
We take $M=M(\epsilon,\delta,\gamma)$ such that
$$
c_1\left(\left(\frac{6\gamma}{(1+\delta)\epsilon}\right)^{\frac{1}{1+\delta-\gamma}}+1\right)\left(\frac{1}{M^{1+\delta}}+\frac{1}{\delta M^{\delta}}\right)^{1/2}\leq\frac{\epsilon}{3},
$$
and we obtain
$$
\frac{1}{2}\frac{d}{dt}\|u(t)-s\|_{L^2(\TT)}^2\leq-2\|u(t)-s\|_{L^2(\TT)}^2+2\|s\|_{L^2(\TT)}^2+\frac{6}{\epsilon}\|\Lambda s\|_{L^2(\TT)}^2.
$$
Using Gronwall inequality, we conclude that
\begin{align*}
\|u(t)-s\|_{L^2(\TT)}^2&\leq\left(\|u_0-s\|_{L^2(\TT)}^2+\|s\|_{L^2(\TT)}^2+\frac{3}{\epsilon}\|\Lambda s\|_{L^2(\TT)}^2\right)e^{-4t}\\
&\quad +\|s\|_{L^2(\TT)}^2+\frac{3}{\epsilon}\|\Lambda s\|_{L^2(\TT)}^2.
\end{align*}
The existence of an absorbing set in $L^2$ is now straightforward. Thus we have the existence of a constant $R=R(\epsilon,\delta,\gamma)$ such that
$$
\|u(t)\|_{L^2(\TT)}\leq R(\epsilon,\delta,\gamma).
$$

\textbf{Step 2: Absorbing set in $L^\infty$} We assume $u(x_t)=\|u(t)\|_{L^\infty(\TT)}$. We take $\nu>0$ a positive number and define
$$
\mathcal{U}_1=\{\eta\in[-\nu,\nu] \;s.t.\; u(x_t)-u(x_t-\eta)>u(x_t)/2 \},
$$
and $\mathcal{U}_2=[-\nu,\nu]-\mathcal{U}_1$. We have
\begin{align*}
R^2(\epsilon,\delta,\gamma)&\geq\|u(t)\|^2_{L^2(\TT)}
\\
&\geq\int_\RR \left(u(x_t-\eta)\right)^2d\eta\\
&\geq\int_{\mathcal{U}_2} \left(u(x_t-\eta)\right)^2d\eta
\\
&\geq \left(\frac{u(x_t)}{2}\right)^2|\mathcal{U}_2|.
\end{align*}
Equivalently,
$$
2\nu-\frac{4R^2}{\|u(t)\|^2_{L^\infty(\TT)}}\leq 2\nu-|\mathcal{U}_2|=|\mathcal{U}_1|.
$$
Using the fact that the initial data has zero mean, we get
\begin{align*}
\Lambda^{1+\delta} u(x_t)&=\sum_{k\in\ZZ}\int_\TT\frac{u(x_t)-u(x_t-\eta)}{|\eta-2k\pi|^{2+\delta}}d\eta\\
&\geq\sum_{|k|>0}\int_{\TT}\frac{u(x_t)-u(x_t-\eta)}{|\eta-2k\pi|^{2+\delta}}d\eta+\int_{\mathcal{U}_1}\frac{u(x_t)-u(x_t-\eta)}{|\eta|^{2+\delta}}d\eta\\
&\geq\sum_{|k|>1}\int_{\TT}\frac{u(x_t)-u(x_t-\eta)}{|2(k-1)\pi|^{2+\delta}}d\eta+\frac{\frac{u(x_t)}{2}}{\nu^2}|\mathcal{U}_1|\\
&\geq \frac{u(x_t)}{\nu^{2+\delta}}\left(\nu-2\left(\frac{R}{u(x_t)}\right)^2\right)+ \frac{2\zeta(2+\delta)u(x_t)}{(2\pi)^{1+\delta}}.
\end{align*}
We define
$$
\nu=3\left(\frac{R}{u(x_t)}\right)^2,
$$
and we obtain
$$
\Lambda^{1+\delta} u(x_t)\geq \frac{\left(u(x_t)\right)^{3+2\delta}}{3^{2+\delta}R^{2(1+\delta)}}+ \frac{2\zeta(2+\delta)u(x_t)}{(2\pi)^{1+\delta}}.
$$
As $\nu\leq\pi$ this choice implies
$$
\sqrt{\frac{3}{\pi}}R\leq u(x_t).
$$
We have
\begin{align*}
\frac{d}{dt}\|u(t)\|_{L^\infty(\TT)}&\leq \Lambda^\gamma u(x_t)-\frac{1}{2}\Lambda^{1+\delta}u(x_t)-\frac{1}{2}\Lambda^{1+\delta}u(x_t)\\
&\leq C(\gamma,\delta)\|u(t)\|_{L^\infty(\TT)}-\frac{1}{2}\left(\frac{\left(u(x_t)\right)^{3+2\delta}}{3^{2+\delta}R^{2(1+\delta)}}+ \frac{2\zeta(2+\delta)u(x_t)}{(2\pi)^{1+\delta}}\right)\\
&\leq C(\gamma,\delta)\|u(t)\|_{L^\infty(\TT)}-\frac{\|u(t)\|_{L^\infty(\TT)}^{3+2\delta}}{2\cdot 3^{2+\delta}R^{2(1+\delta)}}.
\end{align*}
On the other hand, if $\|u(t)\|_{L^\infty(\TT)}=-\min_x u(x,t)$, we define
$$
\mathcal{U}_1=\{\eta\in[-\nu,\nu] \;s.t.\; -u(x_t)+u(x_t-\eta)>-u(x_t)/2 \},
$$
and $\mathcal{U}_2=[-\nu,\nu]-\mathcal{U}_1$. We get
\begin{multline*}
\frac{d}{dt}\|u(t)\|_{L^\infty(\TT)}= -\Lambda^\gamma u(x_t)+\Lambda^{1+\delta}u(x_t)=\Lambda^\gamma (-u(x_t))-\Lambda^{1+\delta}(-u(x_t))\\
\leq C(\gamma,\delta)\|u(t)\|_{L^\infty(\TT)}-\frac{\|u(t)\|_{L^\infty(\TT)}^{3+2\delta}}{2\cdot 3^{2+\delta}R^{2(1+\delta)}}.
\end{multline*}
Collecting these inequalities, we obtain the existence of an absorbing ball in $L^\infty$ with radius
$$
r_\infty=\max\left\{\sqrt{\frac{3}{\pi}}R,C\left(\gamma,\delta\right)R\right\}.
$$

\textbf{Step 3: Absorbing set in $L^p$} For the case $2<p<\infty$, we use interpolation. We get
\begin{align*}
\|u(t)\|_{L^p(\TT)}&\leq \|u(t)\|^{2/p}_{L^2(\TT)}\|u(t)\|^{1-2/p}_{L^\infty(\TT)}\\
&\leq R^{2/p}\max\left\{\sqrt{\frac{3}{\pi}}R,C\left(\delta\right)R,\|u_0\|_{L^\infty(\TT)}\right\}^{1-2/p}.
\end{align*}
The radius for this case can be obtained in a similar way.

\textbf{Step 4: Initial data without odd symmetry} Following the same ideas as in \cite{goodmankuramoto} (see also \cite{colletkuramoto2, frankelroytburd}), we introduce the set of translations of the function $s$ defined in \eqref{sdefi}:
\begin{align*}
\mathcal{S}=\{\tilde{s} : \mbox{$\tilde{s}(x)=s(x+\chi)$ with $|\chi|\leq\pi$}\}.
\end{align*}
Since the function $u_0$ has zero mean, the solution $u(t)$ has zero mean for all time, so there exists at least one point $x_0(t)$ such that $u(x_0(t),t)=0$. Then, for any particular time $t$, we consider, as in the step 1 above, the function $b^{x_0(t)}_M(x)$
defined in Lemma \ref{aux1} where $\lambda$ was defined in \eqref{lambdadefi}, and let
$$
\pax \tilde{s}(x,t)=-2\lambda \sum_{0<|\xi|\leq M}e^{-i\xi (x-x_0(t))} = -2\lambda \left[b^{x_0(t)}_M(x)-1\right]
$$
Notice that $\tilde{s}(x)=s(x+x_0(t))$, with $s$ defined in \eqref{sdefi}. As before, we obtain
\begin{multline*}
\frac{d}{dt'}\|u(t+t')-\tilde{s}(t)\|_{L^2(\TT)}^2\\
\leq-4\|u(t+t')-\tilde{s}(t)\|_{L^2(\TT)}^2+4\|s(t)\|_{L^2(\TT)}^2+\frac{12}{\epsilon}\|\Lambda s(t)\|_{L^2(\TT)}^2.
\end{multline*}
If follows that
$$
\frac{d}{dt'}\|u(t+t')-\tilde{s}(t)\|_{L^2(\TT)}^2\bigg|_{t'=0}\leq 0
$$
if
$$
d(u(t),\tilde{s}(t))=\|u(t)-\tilde{s}(t)\|_{L^2(\TT)} \gg 1.
$$
As a consequence, we find that
$$
d(u(t),\tilde{s}(t))=\|u(t)-\tilde{s}(t)\|_{L^2(\TT)}
$$
is a bounded function of time. Since $d(u(t),\mathcal{S})\leq d(u(t),\tilde{s}(t))$,
this completes the proof.
\end{proof}

\begin{corol}Let $u_0\in H^\alpha(\TT), \alpha>1$ be the mean-zero initial data for the problem \eqref{KS} with $\epsilon\geq1>\delta$ in the periodic case. Then we have
$$
\|u(t)\|_{L^p(\TT)}\leq \|u_0\|^{2/p}_{L^2(\TT)}\max\left\{\sqrt{\frac{3}{\pi}}\|u_0\|_{L^2(\TT)},C\left(\delta\right)\|u_0\|_{L^2(\TT)},\|u_0\|_{L^\infty(\TT)}\right\}^{1-2/p},
$$
and
$$
\limsup_{t\rightarrow\infty}\|u(t)\|_{L^p(\TT)}=\|u_0\|^{2/p}_{L^2(\TT)}\left(\max\left\{\sqrt{\frac{3}{\pi}}\|u_0\|_{L^2(\TT)},C\left(\delta\right)\|u_0\|_{L^2(\TT)}\right\}\right)^{1-2/p}.
$$
\end{corol}
\begin{proof}
The result follows from Poincar\'e's inequality.
\end{proof}

The existence of an absorbing set in the $L^2$-norm and the regularity results from Section \ref{sec:global} imply
the existence of an absorbing set in higher Sobolev norms. The proof is straightforward, and we just state the result.

\begin{lem}\label{higherset}
Suppose that $\alpha>1$ and $u_0\in H^\alpha(\TT)$ has zero mean. Then for every $0<s\leq \alpha$ the solution $u$ of the initial-value problem \eqref{KSivp} in the periodic case satisfies
$$
\limsup_{t\rightarrow\infty}\|u(t)\|_{H^s}\leq
C\left(s,\epsilon,\delta,\gamma,\|u_0\|_{L^2(\TT)}\right) .
$$
\end{lem}

\section{The attractor}
\label{sec:attractor}

In this section we prove the existence of an attractor for spatially periodic solutions ($\Omega = \TT$) and derive some of its properties.
\subsection{Existence}
We denote the solution operators for \eqref{KSivp} by $S(t)$, where $S(t)u_0=u(x,t)$.
The compactness of a nonlinear semigroup, or semiflow, is defined as follows \cite{temambook}.

\begin{defi}
The solution operator $S(t)u_0=u(t,x)$ defines a compact semiflow in $H^s$ if, for every $u_0\in H^s$ the following statements hold:
\begin{itemize}
\item $S(0)u_0=u_0$.
\item for all $t,s,u_0$, the semigroup property hold, \emph{i.e.},
$$
S(t+s)u_0=S(t)S(s)u_0=S(s)S(t)u_0.
$$
\item For every $t>0$, $S(t)$ is continuous (as an operator from $H^s$ to $H^s$).
\item There exists $t_1>0$ such that $S(t_1)$ is a compact operator, \emph{i.e.} for every bounded set $B\subset H^s$, $S(t_1)B\subset H^s$ is a compact set.
\end{itemize}
\end{defi}

It is then is straightforward to use our existence results to prove the following lemma.

\begin{lem}\label{aux2}
%If $\alpha\geq 3$, then the solution operators for \eqref{KSivp} form a compact semiflow
%in $H^\alpha(\TT)$.
Let $u_0\in H^\alpha(\TT)$ for $\alpha\geq 3$ be the initial data for the problem \eqref{KS}. Then $S(t)u_0=u(\cdot, t)$ defines a compact semiflow in $H^\alpha(\TT)$. Moreover $S(t)u_0$ is a continuous map from $[0,T]$ to $H^\alpha(\TT)$ for every initial data $u_0$, {i.e.}, $S(\cdot) u_0 \in C([0,T], H^\alpha)$.
\end{lem}
Now we can apply Theorem 1.1 in \cite{temambook} to obtain the existence of the attractor

\begin{teo}\label{attractor}
In the spatially periodic case with $\Omega = \TT$,
equation \eqref{KS}
has a maximal, connected, compact attractor in the space $H^{\alpha}(\TT)$
for every $\alpha\geq 3$.
\end{teo}
\begin{proof}
The result follows from Lemma \ref{higherset}, where the existence of an absorbing set is proved, and Lemma \ref{aux2}, where the properties of the semigroup are proved.
\end{proof}

\subsection{Number of wild oscillations}
\label{sec:wild}
In this section we obtain a bound for the number of wild oscillations that a solution $u$ can develop. This bound is similar to the bound in \cite{AnalyticityKuramotoGrujic} for the standard KS equation (see also \cite{kukavicaKS}), and splits $\TT$ into a set $I_M$ where $\partial_x u$ is uniformly bounded and a set $R_M$ where $\partial_x u$ may be large but $u$ cannot have too many critical points.   However, our bound is valid for arbitrary initial data while the bound in \cite{AnalyticityKuramotoGrujic} only works for initial data in a neighborhood of a stationary solution.

\begin{teo}
\label{th:wild_oscillations}
Let $u$ be the solution of \eqref{KSivp} for initial data $u_0\in H^3(\TT)$ and define $T > 0$ as in \eqref{time}, \eqref{defk}. Then for every $M > 1$, there exist $\tau_M > 0$ and $I_M, R_M \subset \TT$, where $I_M$ a union of at most $[4\pi/\tau_M]$ open intervals, such that $\TT=I_M\cup R_M$ and the following estimates hold for $T/M<t<T$:
\begin{align*}
|\pax u(x,t)| &\leq \frac{\sqrt{2}\|u_0\|_{L^\infty(\TT)}}{M}\qquad \mbox{ for all $x\in I_M$},
\\
\card\{x \in R_M : \pax u(x,t)=0\} &\leq
\frac{4\pi}{\log 2}\frac{\log\left({M}/{\tau_M}\right)}{\tau_M}.
\end{align*}
An explicit choice for $\tau_M$ is
$$
\tau_M=\frac{1}{M}\left[\frac{\log\left(\mathcal{E}/{c}+1\right)}{3\mathcal{E}}\right],
%\qquad
%k=\left(\|\pax^3 u_0\|^2_{L^2(\TT)}+\frac{1}{\|u_0\|^2_{L^\infty(\TT)}}\right)^3.
$$
where $\mathcal{E}$ is given by \eqref{defE}.
\end{teo}
\begin{proof}
%We will use Theorem \ref{SE} and Lemma \ref{grujic}.
%We have that the solution becomes analytic in the complex strip $\BB_k$. In particular, using \eqref{time} and choosing $M\in\RR$, $M\geq2$, after %time ${T}/{M}$, the width is (see \eqref{Ebonita}, \eqref{width})
From Theorem~\ref{SE}, after time $t > 0$ the solution becomes analytic in a complex strip $\BB_k(t)$. In particular, choosing the parameters $k$, $\lambda$ as in \eqref{defk}, we get from \eqref{width} that the width of the strip after time ${T}/{M}$ is at least
$$
\tau_M =\frac{1}{M}\left[\frac{\log\left({\mathcal{E}}/{c}+1\right)}{3\mathcal{E}}\right].
$$
%where $\lambda$ has been chosen $\lambda=\sqrt{2}\|u_0\|_{L^\infty(\TT)}$ (see Theorem \ref{SE}).
Using Cauchy's integral formula and the definition of $d^\lambda[u]$ in Theorem \ref{SE}, we find that
$$
\|\pax u(t)\|_{L^\infty(\BB_k)}\leq \frac{\|u(t)\|_{L^\infty(\BB_k)}}{\tau_M}\leq \frac{\lambda}{\tau_M},
$$
and an application of Lemma~\ref{grujic} with $\mu = {\lambda}/{M}$ then gives the result.
\end{proof}

Theorem~\ref{th:wild_oscillations} is local in time, but we can apply the result repeatedly to get bounds on the number of oscillations on successive time intervals\[
[{T}/{M},T] \cup [ T + T_1/M, T+T_1] \cup \dots,
\]
where $T_1$ is given by \eqref{time} with $u_0$ replaced by $u(T)$. In view of the uniform $H^3$-bounds on $u(t)$, we can extend the estimates to arbitrarily large times, but there are small gaps between successive time intervals in which the estimates may not apply.

\section{Numerical simulations}
\label{sec:numerics}
In this section, we show some numerical solutions of \eqref{KS}, which we repeat here for convenience
\begin{equation}\label{KSnum}
\pat u+\pax\left(\frac{1}{2} u^2\right)= \Lambda^\gamma u -\epsilon\Lambda^{1+\delta}u,
\end{equation}
with $2\pi$-periodic boundary conditions. We approximate the spatial part
by a pseudo-spectral scheme, typically using $2^{12}$--$2^{14}$ Fourier modes, and advance in time with an explicit method such as the \texttt{ode45} function in MATLAB.

In Figures~\ref{surf_evoepsilon}--\ref{evoepsilon}, we show a numerical solution of \eqref{KSnum} with $\delta=\gamma=1$ in $-\pi < x <\pi$ for initial data
\begin{equation}
u_0(x) =\cos x + e^{-x^2}\sin x.%{ (Figure }\ref{evoepsilon}).
\label{ic2}
\end{equation}
A primary ``viscous shock'' forms from the initial data, after which smaller ``viscous sub-shocks'' develop spontaneously throughout the interval. These sub-shocks grow, propagate toward the primary shock, and merge with it. The number of sub-shocks and their rate of formation increases as $\epsilon$ decreases. Some movies of the numerical simulations are available at

\small{
\begin{verbatim}http://youtu.be/8r0QMgxZJMk?list=PLUwnEWNEnlmhroc7JS_cZ2PLN6pe-HiX7\end{verbatim}}

\begin{figure}[h!]\centering
\includegraphics[scale=0.75]{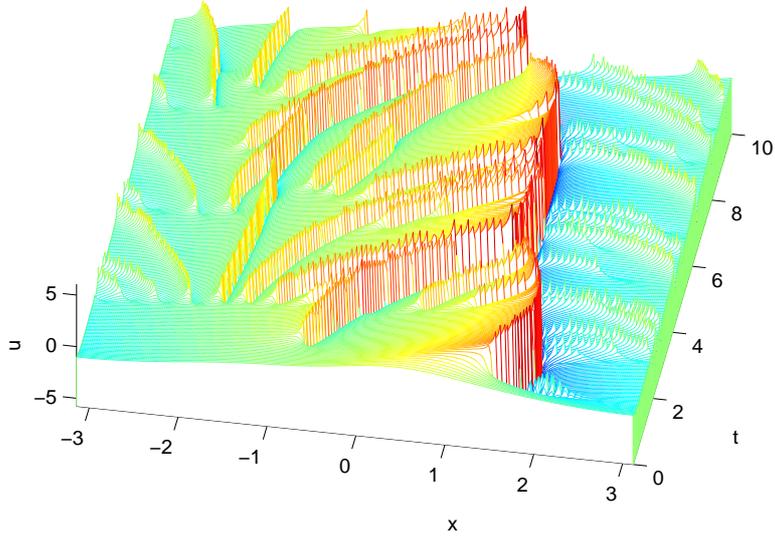}
\caption{A spatially periodic numerical solution of the nonlocal KS equation \eqref{KSnum} with $\delta=1$, $\gamma=1$, $\epsilon=0.01$, and initial data \eqref{ic2}.}
\label{surf_evoepsilon}
\end{figure}

\begin{figure}[h!]\centering
\includegraphics[scale=0.25]{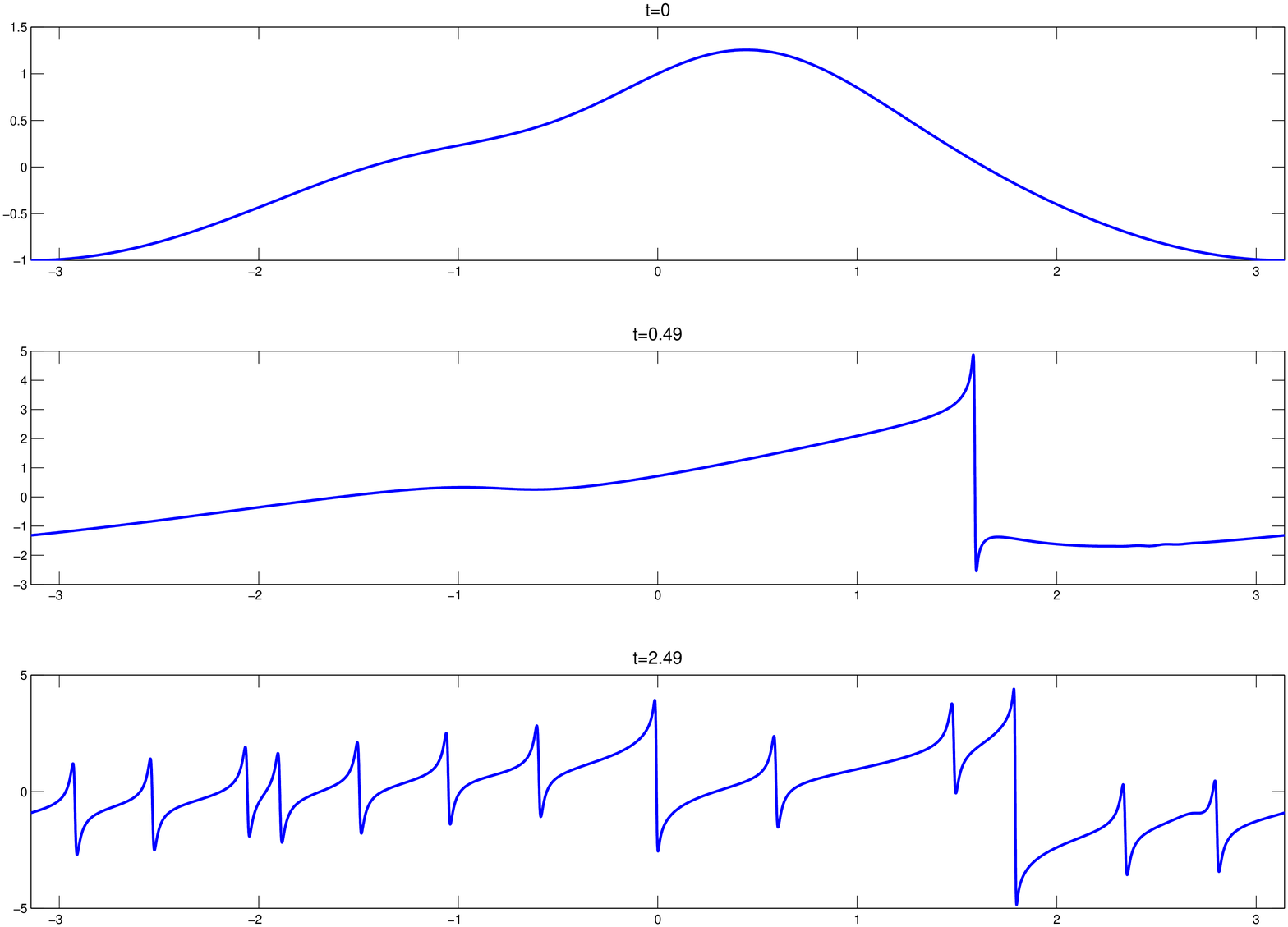}
\caption{A numerical solution of  \eqref{KSnum} with $\delta=1$, $\gamma=1$, $\epsilon=0.01$ and the same initial data as in Figure~\ref{surf_evoepsilon} at $t=0, 0.49, 2.49$.}
\label{evoepsilon}
\end{figure}

In Figure~\ref{surf_ks}, we show a solution of the usual KS equation \eqref{usualKS} with the same initial data as in Figure~\ref{surf_evoepsilon}.
The spatial ``shock-like'' structure of chaotic solutions of \eqref{KSnum} is qualitatively different from the ``worm-like'' structure of solutions of \eqref{usualKS}.

\begin{figure}[h!]\centering
\includegraphics[scale=0.75]{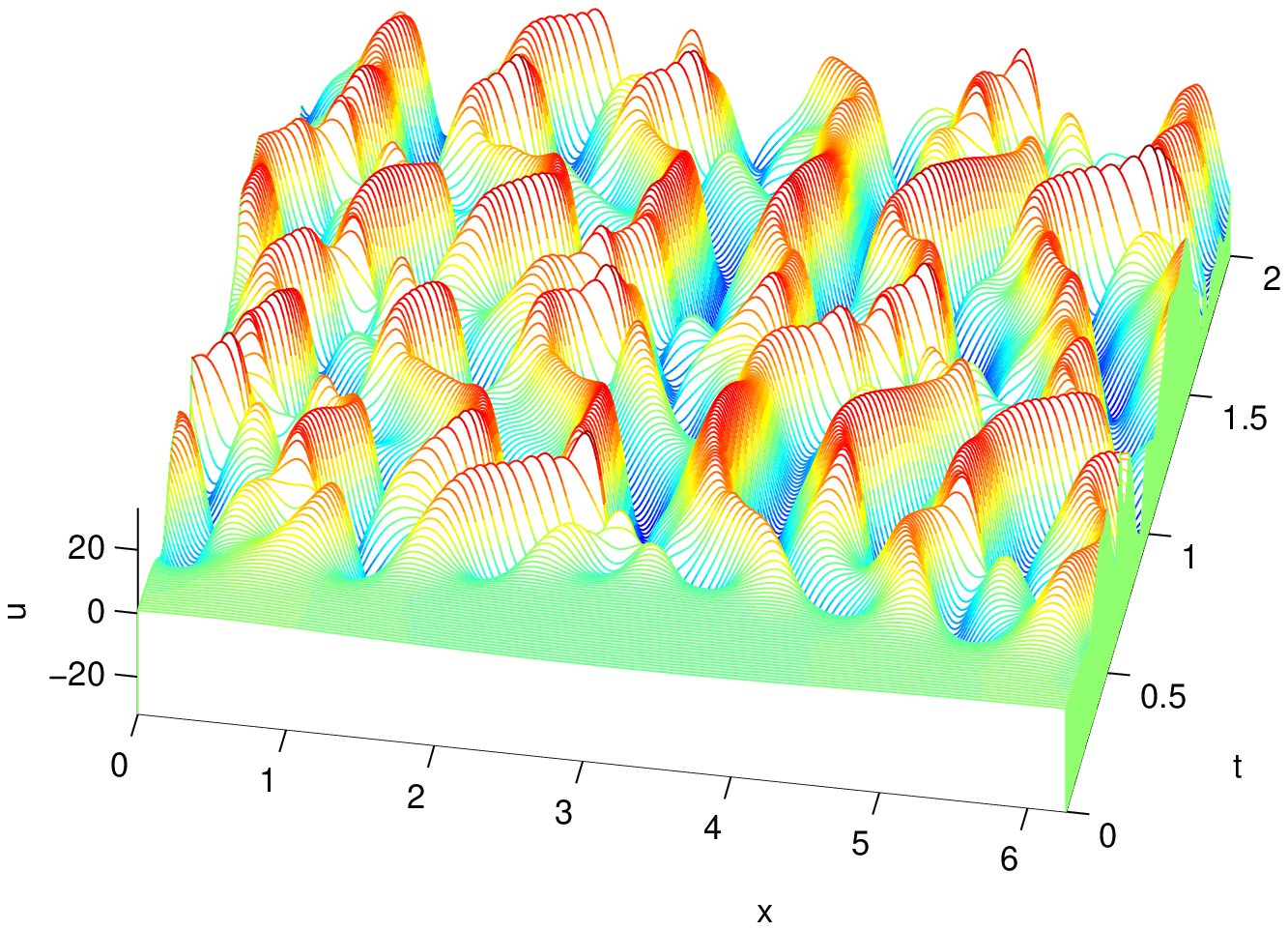}
\caption{A numerical solution of the usual KS equation \eqref{usualKS} with $\epsilon=0.01$ and initial data \eqref{ic2}.}
\label{surf_ks}
\end{figure}

Similar behavior is observed for \eqref{KSnum} with other values of $0 < \delta < 1$, $0 < \gamma<1+\delta$, and $\epsilon > 0$.
In Figure~\ref{evogamma} we show a solution for $\delta=0.5$, $\gamma=1.45$, and $\epsilon=0.8$, with the initial data
\begin{equation}
u_0(x) =\cos x.%\text{ (Figure }\ref{evogamma}),
\label{ic1}
\end{equation}
Chaotic behaviour occurs for larger values of $\epsilon$ as $\gamma$ gets closer to $1+\delta$.
This is consistent with the fact that the band of unstable wavenumbers $k$ for the linearization of \eqref{KSnum} at $u=0$
is given by
\[
0<k<k_*(\delta,\gamma,\epsilon)\qquad \mbox{where}\qquad \epsilon k_*^{1+\delta-\gamma} = 1.
\]
Thus, for a fixed value of $\epsilon$, the unstable band gets wider as $\gamma$ increases toward $1+\delta$.
(We have $k_\ast = 100$ in Figure~\ref{evoepsilon} and $k_\ast \approx 87$ in Figure~\ref{evogamma}.)

\begin{figure}[h!]\centering
\includegraphics[scale=0.25]{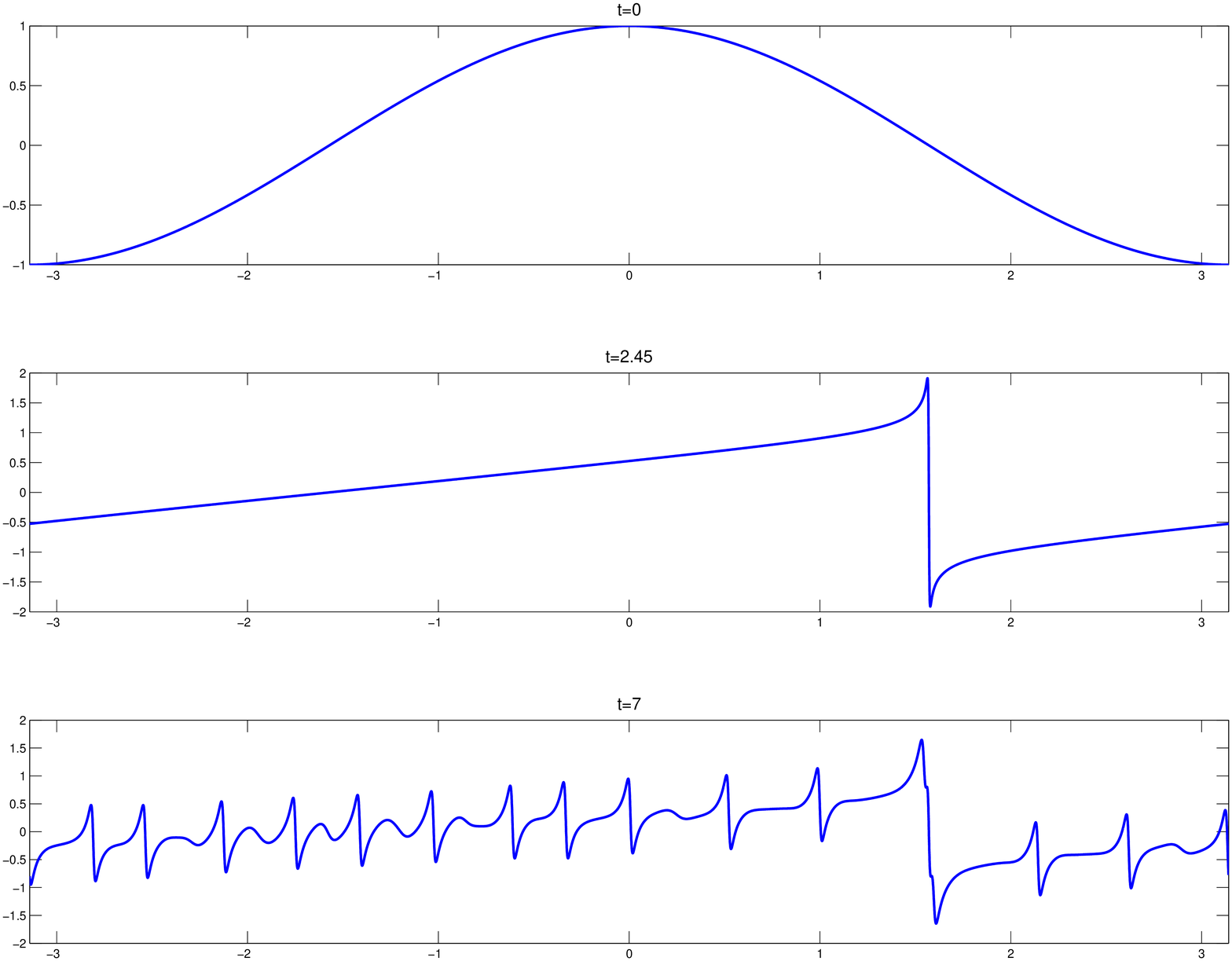}
\caption{A solution of \eqref{KSnum} with $\delta=0.5$, $\gamma=1.45$, $\epsilon=0.8$, and initial data \eqref{ic1}
at $t=0, 2.45, 7$.}
\label{evogamma}
\end{figure}

%Let's write $u^\gamma$ ($u^\epsilon$) for the solution corresponding to certain value of the parameters.

%In Figures~\ref{Gamma1}--\ref{Gamma2}, we show the transition to chaos for $\epsilon=0.5$, $\delta=0.5$ as $\gamma$ increases toward $1.5$ by plotting the large time behaviour of the $L^\infty$ and $L^2$ norms of $u^\gamma$. In Figures~\ref{epsilon1}--\ref{epsilon2}, we show the transition to chaos for $\gamma=1$, $\delta=1$ as $\epsilon$ decreases toward $0$.

%To do that, we compute $\|u^\gamma(t)\|_{L^2},\|u^\gamma(t)\|_{L^\infty}$ $(\|u^\epsilon(t)\|_{L^2},\|u^\epsilon(t)\|_{L^\infty}$) for several values of $\gamma$ ($\epsilon$) and large times. Then, on the vertical axis, for a given $\gamma$,($\epsilon$), we represent $\|u^\gamma(t)\|_{L^2},\|u^\gamma(t)\|_{L^\infty}$ $(\|u^\epsilon(t)\|_{L^2},\|u^\epsilon(t)\|_{L^\infty}$) for different times. So, we see that for certain values of $\gamma$ ($\epsilon$) the solution approaches a steady solution but when there parameters are in a different range the solutions behave wildly.

Figures~\ref{Gamma1}--\ref{Gamma3} show the transition to chaos for $\epsilon=0.5$, $\delta=0.5$ as $\gamma$ increases toward $1.5$. For each value of $\gamma$, we plot the $L^\infty$ and $L^2$ norms of $u$ at a number of different times after the solution has approached its time-asymptotic state. %Typically, we integrated the equation up to time $t\approx 10$ and plot the norms for $10 \le t \le 15$.
For $\gamma \lesssim 1.3$ the solution is steady, but for $\gamma \gtrsim 1.3$ its norms fluctuate wildly in time. We have $k_*\approx 32$ at transition.

Similarly, in Figures~\ref{epsilon1}--\ref{epsilon3}, we show the transition to chaos for $\delta=1$, $\gamma=1$ as $\epsilon$ decreases toward $0$. The solution
is steady for $\epsilon \gtrsim 0.04$ and chaotic for $\epsilon \lesssim 0.04$, with $k_\ast \approx 25$ at transition.

%To do that, we compute $\|u^\gamma(t)\|_{L^2},\|u^\gamma(t)\|_{L^\infty}$ $(\|u^\epsilon(t)\|_{L^2},\|u^\epsilon(t)\|_{L^\infty}$) for several values of $\gamma$ ($\epsilon$) and large times. Then, on the vertical axis, for a given $\gamma$,($\epsilon$), we represent $\|u^\gamma(t)\|_{L^2},\|u^\gamma(t)\|_{L^\infty}$ $(\|u^\epsilon(t)\|_{L^2},\|u^\epsilon(t)\|_{L^\infty}$) for different times. So, we see that for certain values of $\gamma$ ($\epsilon$) the solution approaches a steady solution but when there parameters are in a different range the solutions behave wildly.

\begin{figure}[h!]\centering
\includegraphics[scale=0.35]{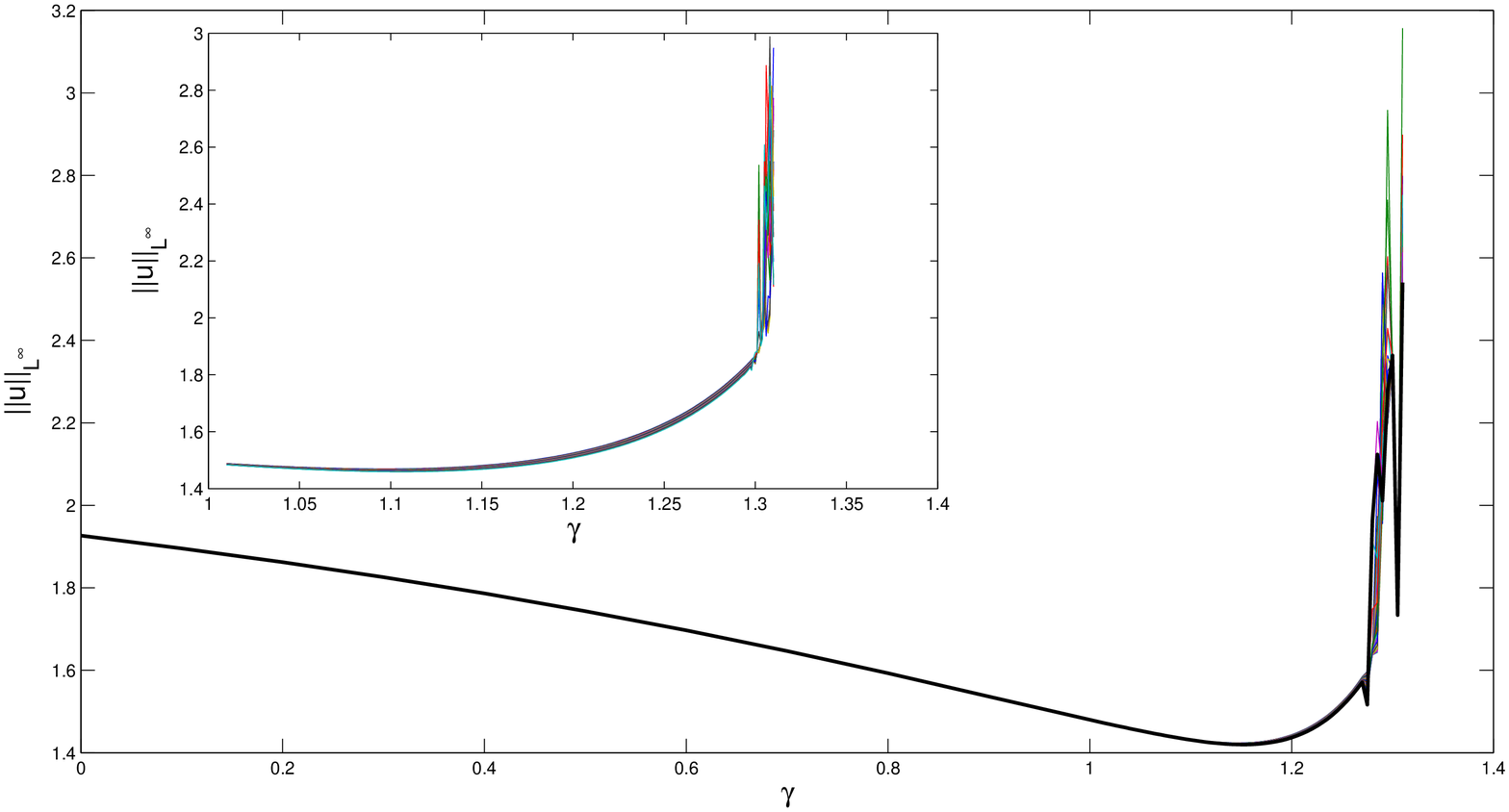}
\caption{The large time behavior of $\|u\|_{L^\infty}$ for different values of $\gamma\in (1,1.4)$ with $\delta=0.5$, $\epsilon=0.5$.}
\label{Gamma1}
\end{figure}
\begin{figure}[h!]\centering
\includegraphics[scale=0.35]{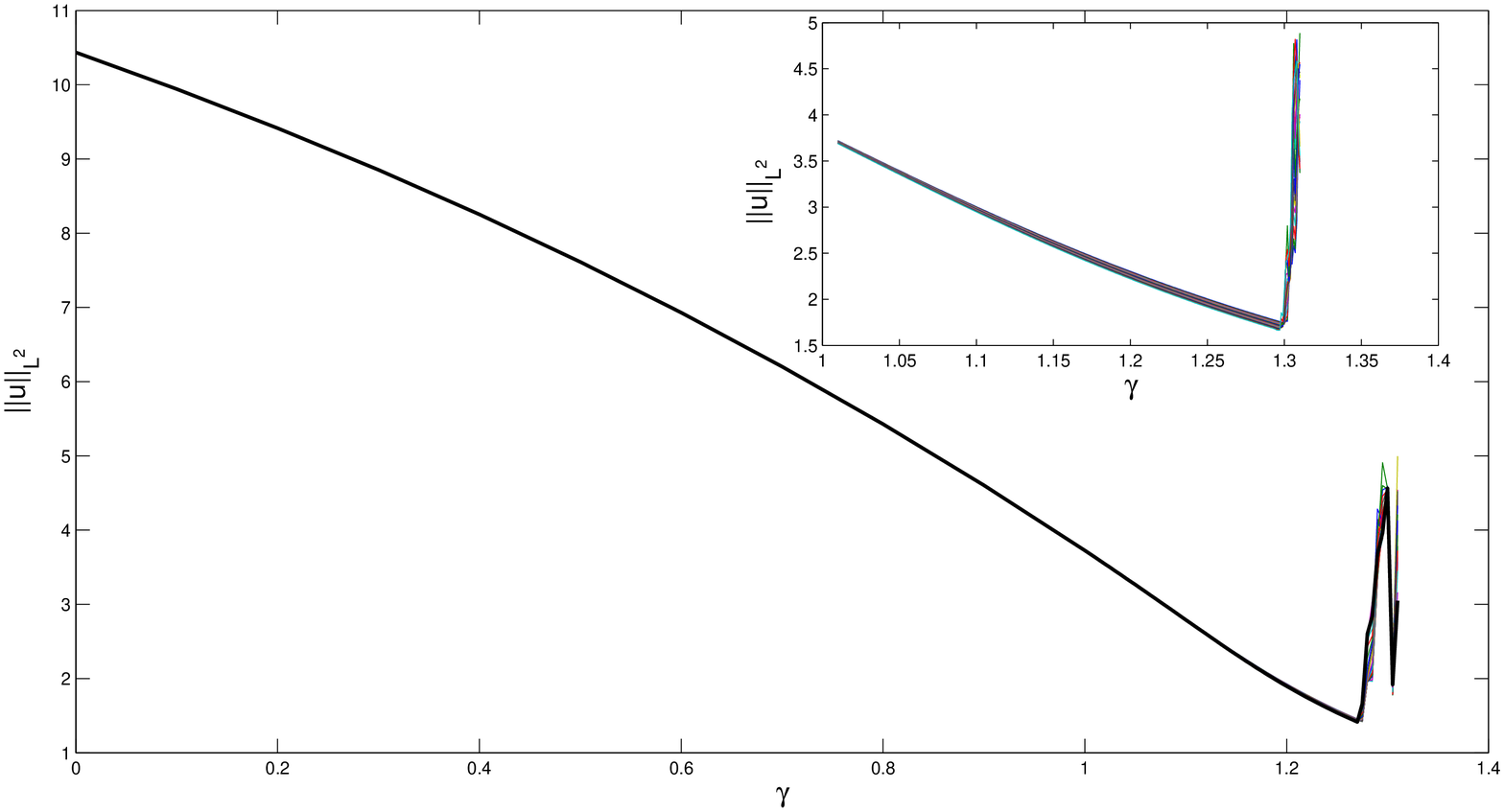}
\caption{The large time behavior of $\|u\|_{L^2}$ for different values of $\gamma\in (1,1.4)$ with $\delta=0.5$, $\epsilon=0.5$.}
\label{Gamma2}
\end{figure}
\begin{figure}[h!]\centering
\includegraphics[scale=0.35]{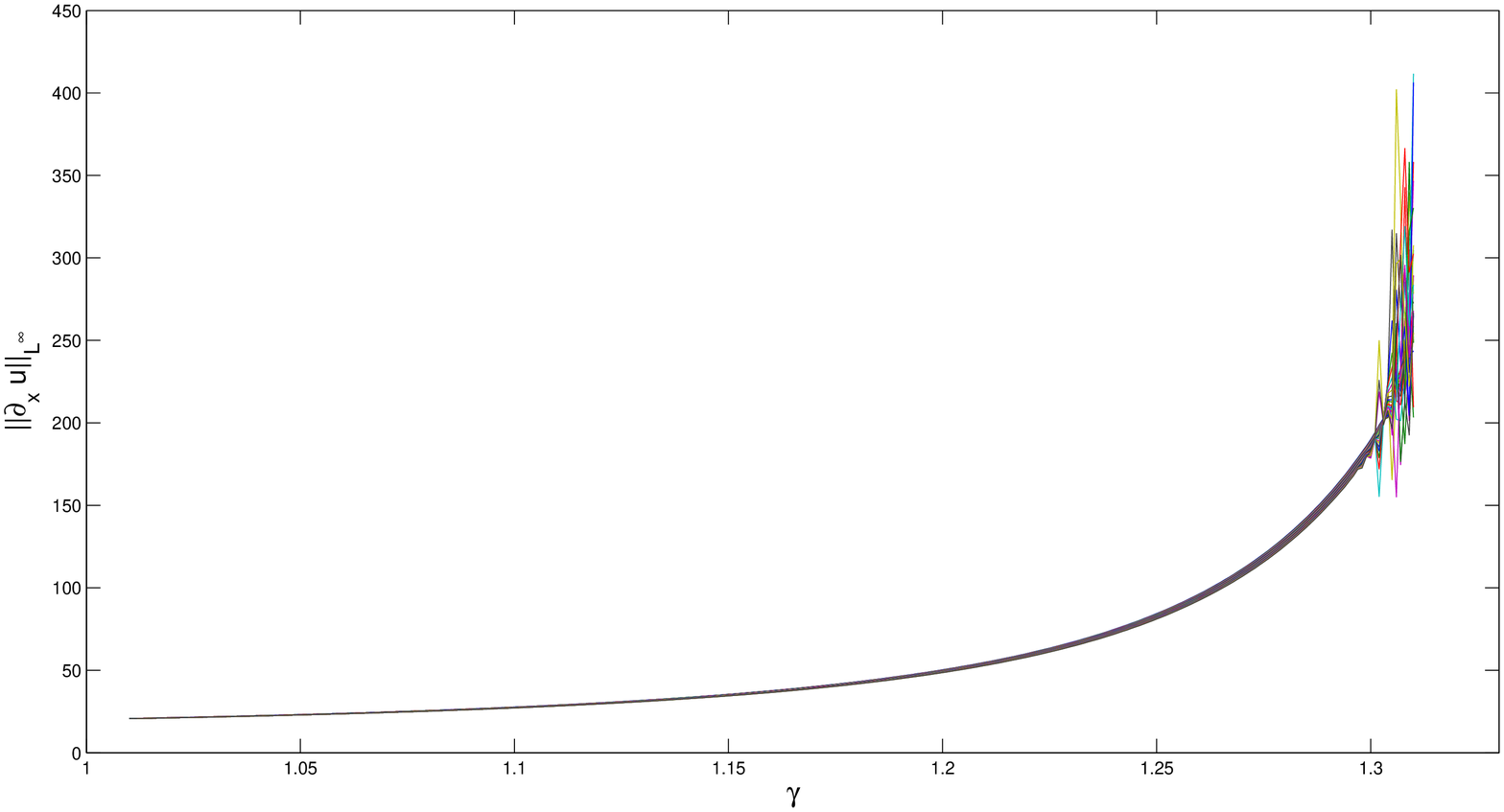}
\caption{The large time behavior of $\|\partial_x u\|_{L^\infty}$ for different values of $\gamma\in (1,1.4)$ with $\delta=0.5$, $\epsilon=0.5$.}
\label{Gamma3}
\end{figure}
\begin{figure}[h!]\centering
\includegraphics[scale=0.35]{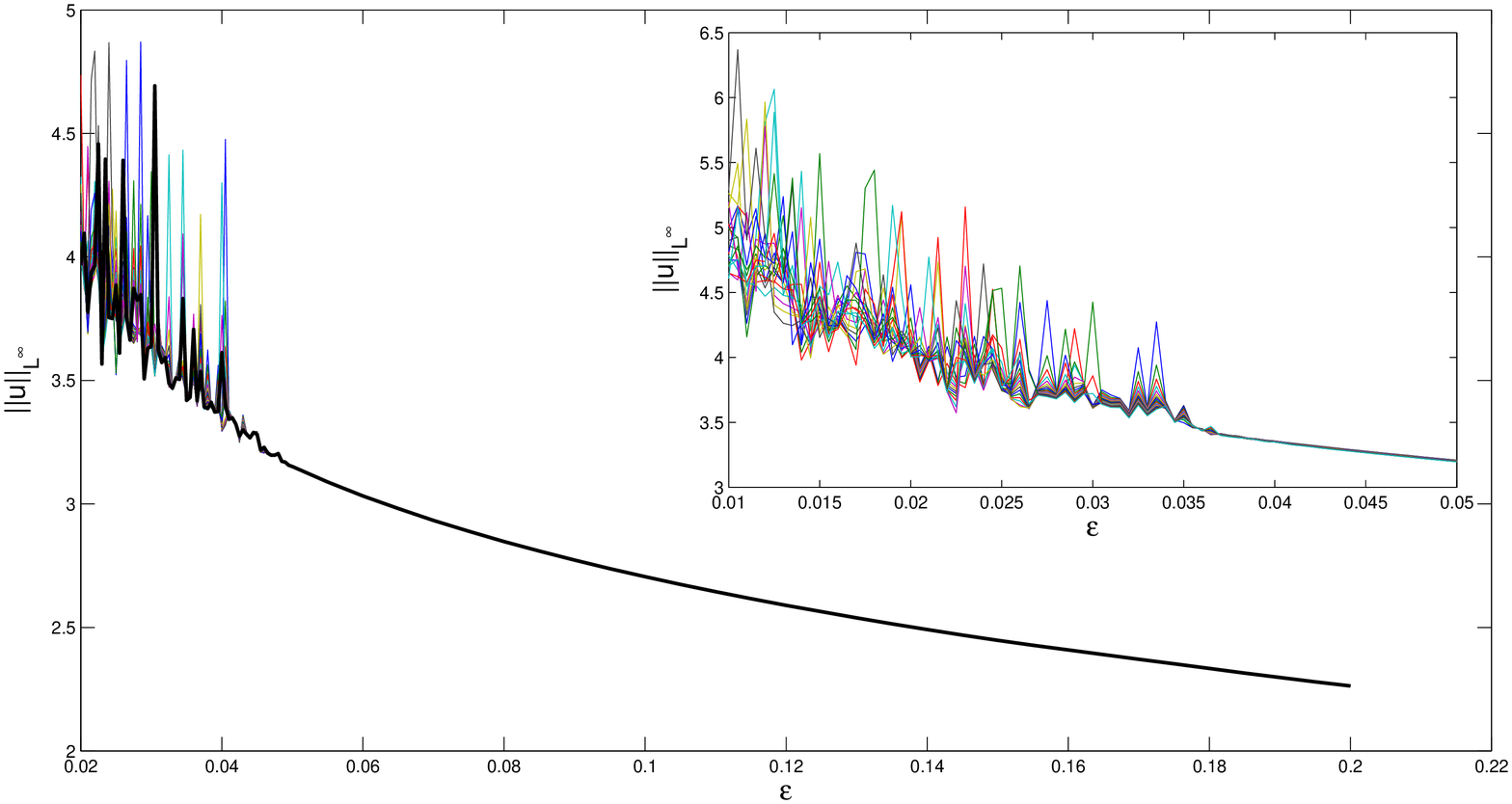}
\caption{The large time behaviour of $\|u\|_{L^\infty}$ for different values of $\epsilon\in(0.02,0.2)$ with $\delta=1$, $\gamma=1$.}
\label{epsilon1}
\end{figure}
\begin{figure}[h!]\centering
\includegraphics[scale=0.35]{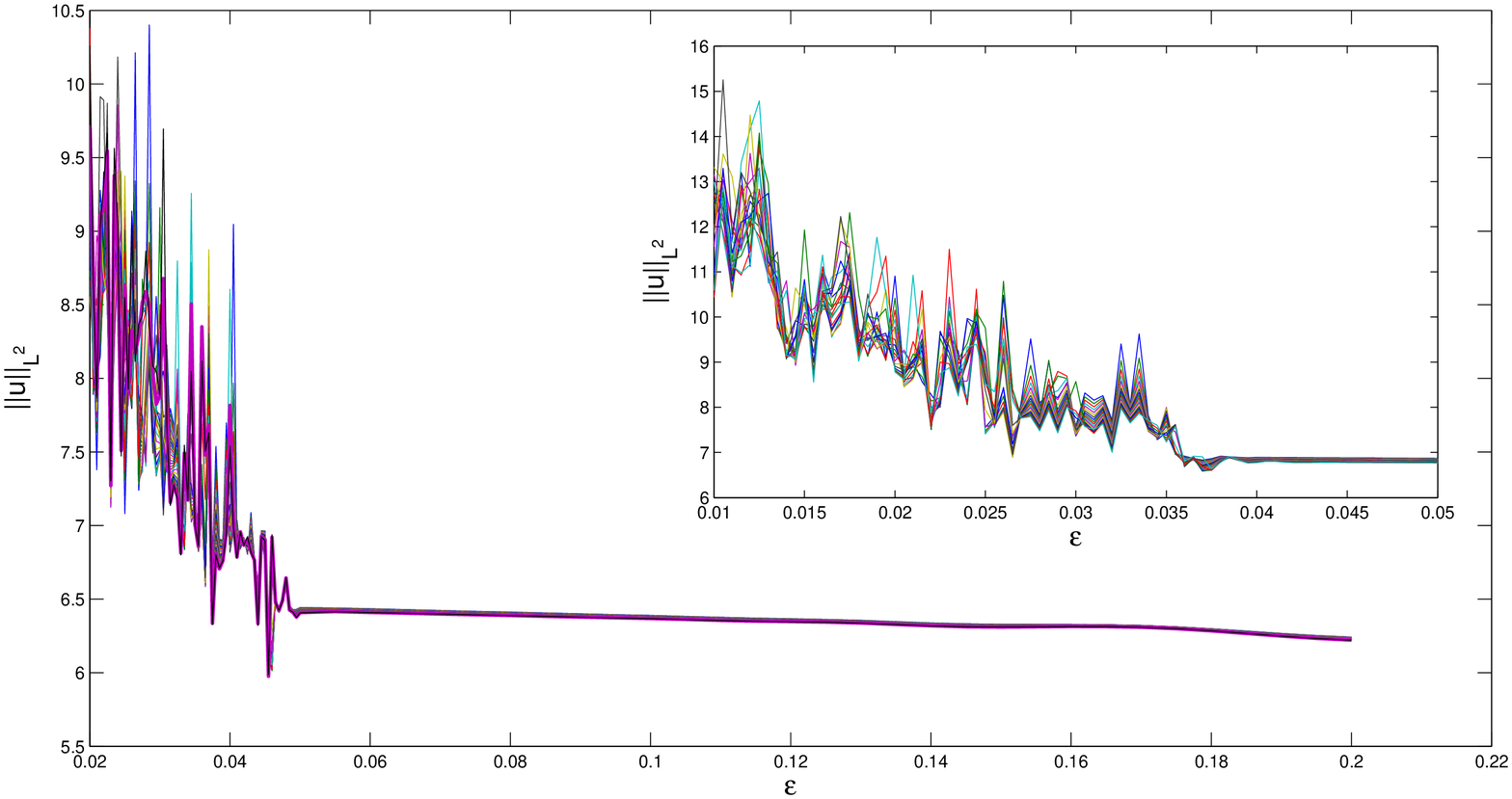}
\caption{The large time behaviour of $\|u\|_{L^2}$ for different values of $\epsilon\in(0.02,0.2)$ with $\delta=1$, $\gamma=1$.}
\label{epsilon2}
\end{figure}
\begin{figure}[h!]\centering
\includegraphics[scale=0.35]{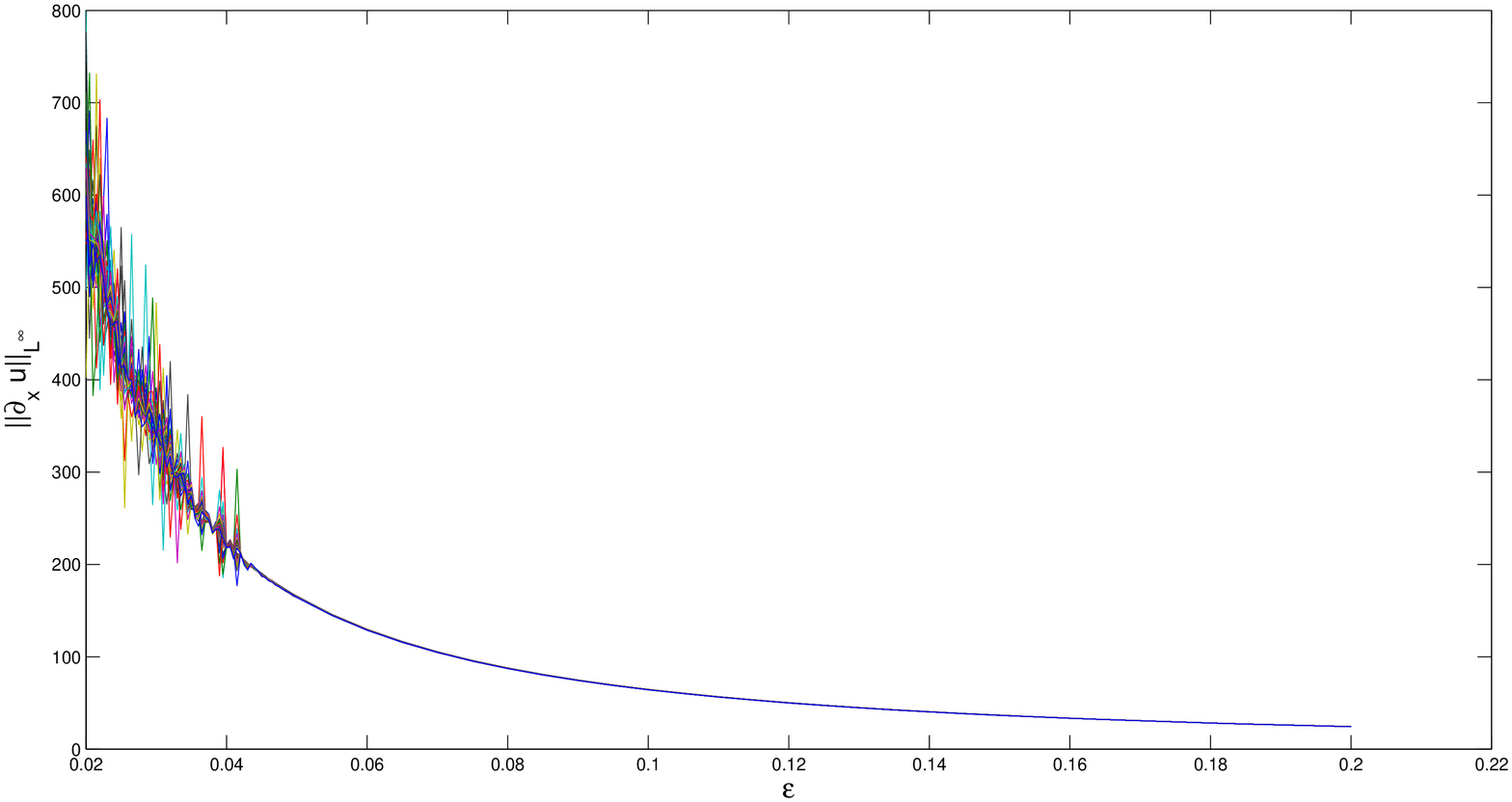}
\caption{The large time behaviour of $\|\partial_x u\|_{L^\infty}$ for different values of $\epsilon\in(0.02,0.2)$ with $\delta=1$, $\gamma=1$.}
\label{epsilon3}
\end{figure}

\appendix

\section{Auxiliary results}
In this appendix, we state without proof several results used in the paper.

We start with the Kato-Ponce inequality and the Kenig-Ponce-Vega commutator estimate for $[\Lambda^s,F]=\Lambda^s F-F\Lambda^s$, where $\Lambda=\sqrt{-\partial_x^2}$ (see \cite{grafakos2013kato, kato1988commutator, KenigPonceVega}).

\begin{lem}\label{commutator}
Let $F$, $G$ be two smooth functions that decay at infinity. Then, for $0<s\leq1$, we have
\begin{align*}
\|\Lambda^s(FG)-F\Lambda^s G\|_{L^p}&\leq C\left(\|F\|_{W^{s,p_1}}\|G\|_{L^{p_2})}\right.\\
&\quad\left.+\|G\|_{W^{s-1,p_3}}\|\pax F\|_{L^{p_4}}\right),
\end{align*}
with
$$
\frac{1}{p}=\frac{1}{p_1}+\frac{1}{p_2}=\frac{1}{p_3}+\frac{1}{p_4}\qquad \mbox{where $1\le p_2,p_4\le\infty$, $1<p, p_1,p_3 <\infty$}.
$$
Furthermore, if $s>\max\{0,1/p-1\}$, then
$$
\|\Lambda^s(FG)\|_{L^p(\RR)}\leq C\left(\|\Lambda^s F\|_{L^{p_1}(\RR)}\|G\|_{L^{p_2}(\RR)}\right.\\
\left.+\|\Lambda^s G\|_{L^{p_3}(\RR)}\|F\|_{L^{p_4}(\RR)}\right),
$$
with
$$
\frac{1}{p}=\frac{1}{p_1}+\frac{1}{p_2}=\frac{1}{p_3}+\frac{1}{p_4}\qquad \mbox{where $1/2<p<\infty,1<p_i\leq\infty$}.
$$
\end{lem}
We require the following uniform Gronwall lemma (see \cite{temambook}).
\begin{lem}\label{UGL}
Suppose that $g$, $h$, $y$ are non-negative, locally integrable functions on $(0,\infty)$ and ${dy}/{dt}$ is locally integrable. If
there are positive constants $a_1$, $a_2$, $a_3$, $r$ such that
$$
\frac{dy}{dt}\leq gy+h,\quad
\int_t^{t+r} g(s)ds\leq a_1,\;\int_t^{t+r} h(s)ds\leq a_2,\;\int_t^{t+r} y(s)ds\leq a_3
$$
for $t \ge 0$, then
$$
y(t+r)\leq \left(\frac{a_3}{r}+a_2\right)e^{a_1}.
$$
\end{lem}
We also use the following result on the time derivative of a complex function  (see \cite{ccfgl}).
\begin{lem}\label{complexevolution}
Suppose that $h(t)>0$ is a decreasing, smooth function of $t$, and
\[
\phi(x\pm i\zeta,t)=\sum_{|\xi|\leq N}A_{\xi}(t)e^{i\xi(x\pm i\zeta)}.
\]
Then
\begin{align*}
&\pat \sum_{\pm}\int_\TT|\phi(x\pm i\zeta,t)|^2dx\\
&\qquad\leq \frac{h'(t)}{10}\sum_{\pm}\int_\TT\Lambda\phi(x\pm i\zeta,t)\overline{\phi(x\pm i\zeta,t)}dx\\
&\qquad\quad-10h'(t)\sum_{\pm}\int_\TT\Lambda\phi(x,t)\overline{\phi(x,t)}dx\\
&\qquad\quad+2\Re\sum_{\pm}\int_\TT\pat\phi(x\pm i\zeta,t)\overline{\phi(x\pm i\zeta,t)}dx
\end{align*}
\end{lem}
The last Lemma concerns the number of wild spatial oscillations of an analytic function (see \cite{AnalyticityKuramotoGrujic} and the references therein)
\begin{lem}\label{grujic}
Let $L$, $\tau>0$, and let $u$ be analytic in the neighborhood of $\{z: |\Im z|\leq \tau\}$ and $L$-periodic in the $x$-direction. Then, for any $\mu>0$, $[0, L]=I_\mu\cup R_\mu$, where $I_\mu$ is an union of at most $[\frac{2L}{\tau}]$ intervals open in $[0, L]$, and
\begin{itemize}
\item $|\pax u(x)| \leq \mu, \text{ for all }x\in I_\mu,$
\item $\card\{x \in R_\mu : \pax u(x)=0\}\leq
\frac{2}{\log 2}\frac{L}{\tau}\log\left(\frac{\max_{|\Im z|\leq \tau}|\pax u(z)|}{\mu}\right).$
\end{itemize}
\end{lem}

\bibliographystyle{abbrv}
%\bibliography{bibliografia}

\begin{thebibliography}{10}

\bibitem{ArioliKochKuramoto}
G.~Arioli and H.~Koch.
\newblock Computer-assisted methods for the study of stationary solutions in
  dissipative systems, applied to the {K}uramoto-{S}ivashinski equation.
\newblock {\em Arch. Ration. Mech. Anal.}, 197(3):1033--1051, 2010.

\bibitem{AGM}
Y.~Ascasibar, R.~Granero-Belinch\'on, and J.~M. Moreno.
\newblock An approximate treatment of gravitational collapse.
\newblock {\em Physica D: Nonlinear Phenomena}, 262:71 -- 82, 2013.

\bibitem{NonlocalKuramotoBronski}
J.~C. Bronski, R.~C. Fetecau, and T.~N. Gambill.
\newblock A note on a non-local {K}uramoto-{S}ivashinsky equation.
\newblock {\em Discrete Contin. Dyn. Syst.}, 18(4):701--707, 2007.

\bibitem{BronskiKuramoto}
J.~C. Bronski and T.~N. Gambill.
\newblock Uncertainty estimates and {$L_2$} bounds for the
  {K}uramoto-{S}ivashinsky equation.
\newblock {\em Nonlinearity}, 19(9):2023--2039, 2006.

\bibitem{ccfgl}
A.~Castro, D.~Cordoba, C.~Fefferman, F.~Gancedo, and M.~Lopez-Fernandez.
\newblock {R}ayleigh-{T}aylor breakdown for the {M}uskat problem with
  applications to water waves.
\newblock {\em Annals of Math}, 175:909--948, 2012.

\bibitem{colletkuramoto1}
P.~Collet, J.-P. Eckmann, H.~Epstein, and J.~Stubbe.
\newblock Analyticity for the {K}uramoto-{S}ivashinsky equation.
\newblock {\em Phys. D}, 67(4):321--326, 1993.

\bibitem{colletkuramoto2}
P.~Collet, J.-P. Eckmann, H.~Epstein, and J.~Stubbe.
\newblock A global attracting set for the {K}uramoto-{S}ivashinsky equation.
\newblock {\em Comm. Math. Phys.}, 152(1):203--214, 1993.

\bibitem{cor2}
A.~C{\'o}rdoba and D.~C{\'o}rdoba.
\newblock A maximum principle applied to quasi-geostrophic equations.
\newblock {\em Communications in Mathematical Physics}, 249(3):511--528, 2004.

\bibitem{c-g09}
D.~C{\'o}rdoba and F.~Gancedo.
\newblock {A} maximum principle for the {M}uskat problem for fluids with
  different densities.
\newblock {\em Communications in Mathematical Physics}, 286(2):681--696, 2009.

\bibitem{CGO}
D.~C\'ordoba, R.~Granero-Belinch\'on, and R.~Orive.
\newblock {O}n the confined {M}uskat problem: differences with the deep water
  regime.
\newblock {\em Communications in Mathematical Sciences}, 12(3):423-455, 2014.

\bibitem{Valdinoci1}
E.~Di~Nezza, G.~Palatucci, and E.~Valdinoci.
\newblock {H}itchhikerʼs guide to the fractional {S}obolev spaces.
\newblock {\em Bulletin des Sciences Math{\'e}matiques}, 136(5):521--573, 2012.

\bibitem{NonlocalKuramotoDuan}
J.~Duan, V.~J. Ervin, and H.~Gao.
\newblock Trajectory and attractor convergence for a nonlocal
  {K}uramoto-{S}ivashinsky equation.
\newblock {\em Comm. Appl. Nonlinear Anal.}, 5(4):33--40, 1998.

\bibitem{evans} L.~C.~Evans and R.~Gariepy.
\newblock \emph{Measure Theory and Fine Properties of Functions}
\newblock CRC Press, 1991

\bibitem{FerrariTiti}
A.~Ferrari, and E.~Titi.
\newblock Gevrey regularity for nonlinear analytic parabolic equations.
\newblock {\em Communications in PDE.}, 23(1-2):1--16, 1998.


\bibitem{Figueras-DeLaLLave:cap-periodic-orbits-kuramoto}
J.-L. Figueras and R.~de~la Llave.
\newblock Numerical computation and a-posteriori verification of periodic
  orbits of the {K}uramoto-{S}ivashinsky equation.
\newblock 2013.
\newblock Preprint.

\bibitem{FNSTkuramoto}
C.~Foias, B.~Nicolaenko, G.~R. Sell, and R.~Temam.
\newblock Inertial manifolds for the {K}uramoto-{S}ivashinsky equation and an
  estimate of their lowest dimension.
\newblock {\em J. Math. Pures Appl.}, 67(3):197--226, 1988.


\bibitem{FTkuramoto2}
C.~Foias, and R.~Temam.
\newblock Gevrey class regularity for the solutions of the Navier-Stokes equations.
\newblock {\em J. Funct. Anal.}, 87: 359--369, 1989.


\bibitem{frankelroytburd}
M.~Frankel and V.~Roytburd.
\newblock Dissipative dynamics for a class of nonlinear pseudo-differential
  equations.
\newblock {\em J. Evol. Equ.}, 8(3):491--512, 2008.

\bibitem{giacomelliottokuramoto}
L.~Giacomelli and F.~Otto.
\newblock New bounds for the {K}uramoto-{S}ivashinsky equation.
\newblock {\em Comm. Pure Appl. Math.}, 58(3):297--318, 2005.

\bibitem{goodmankuramoto}
J.~Goodman.
\newblock Stability of the {K}uramoto-{S}ivashinsky and related systems.
\newblock {\em Comm. Pure Appl. Math.}, 47(3):293--306, 1994.

\bibitem{grafakos2013kato}
L.~Grafakos and S.~Oh.
\newblock {T}he {K}ato-{P}once {I}nequality.
\newblock {\em arXiv preprint arXiv:1303.5144}, 2013.

\bibitem{AnalyticityKuramotoGrujic}
Z.~Gruji{\'c}.
\newblock Spatial analyticity on the global attractor for the
  {K}uramoto-{S}ivashinsky equation.
\newblock {\em J. Dynam. Differential Equations}, 12(1):217--228, 2000.

\bibitem{timeanalyticityKuramoto}
Z.~Gruji{\'c} and I.~Kukavica.
\newblock A remark on time-analyticity for the {K}uramoto-{S}ivashinsky
  equation.
\newblock {\em Nonlinear Anal.}, 52(1):69--78, 2003.

\bibitem{kato1988commutator}
T.~Kato and G.~Ponce.
\newblock Commutator estimates and the {E}uler and {N}avier-{S}tokes equations.
\newblock {\em Communications on Pure and Applied Mathematics}, 41(7):891--907,
  1988.

\bibitem{KenigPonceVega}
C.~E. Kenig, G.~Ponce, and L.~Vega.
\newblock Well-posedness and scattering results for the generalized korteweg-de
  vries equation via the contraction principle.
\newblock {\em Communications on Pure and Applied Mathematics}, 46(4):527--620,
  1993.

\bibitem{kiselevburgers}
A.~Kiselev, F.~Nazarov, and R.~Shterenberg.
\newblock Blow up and regularity for fractal {B}urgers equation.
\newblock {\em Dyn. Partial Differ. Equ.}, 5(3):211--240, 2008.

\bibitem{kukavicaKS}
I.~Kukavica.
\newblock Oscillations of solutions of the {K}uramoto-{S}ivashinsky equation.
\newblock {\em Phys. D}, 76(4):369--374, 1994.

\bibitem{kuramoto1976persistent}
Y.~Kuramoto and T.~Tsuzuki.
\newblock Persistent propagation of concentration waves in dissipative media
  far from thermal equilibrium.
\newblock {\em Progress of theoretical physics}, 55(2):356--369, 1976.

\bibitem{lee} Y.~C.Lee and H.~H.~Chen, Nonlinear models of plasma turbulence,
\emph{Physica Scripta}, T2(1):41--47, 1982.

\bibitem{bertozzi-Majda}
A.~Majda and A.~Bertozzi.
\newblock {\em {V}orticity and incompressible flow}.
\newblock Cambridge Univ Pr, 2002.

\bibitem{NSTkuramoto}
B.~Nicolaenko, B.~Scheurer, and R.~Temam.
\newblock Some global dynamical properties of the {K}uramoto-{S}ivashinsky
  equations: nonlinear stability and attractors.
\newblock {\em Phys. D}, 16(2):155--183, 1985.

\bibitem{NSTkuramoto2}
B.~Nicolaenko, B.~Scheurer, and R.~Temam.
\newblock Attractors for the {K}uramoto-{S}ivashinsky equations.
\newblock In {\em Nonlinear systems of partial differential equations in
  applied mathematics, {P}art 2 ({S}anta {F}e, {N}.{M}., 1984)}, volume~23 of
  {\em Lectures in Appl. Math.}, pages 149--170. Amer. Math. Soc., Providence,
  RI, 1986.

\bibitem{ott} E.~Ott and R.~N.~Sudan, Nonlinear theory of ion acoustic waves with Landau damping,
\emph{Phys. Fluids}, 1969 (12).

\bibitem{ottokuramoto}
F.~Otto.
\newblock Optimal bounds on the {K}uramoto-{S}ivashinsky equation.
\newblock {\em J. Funct. Anal.}, 257(7):2188--2245, 2009.

\bibitem{simon1986compact}
J.~Simon.
\newblock {C}ompact sets in the space $l^{p}(o, t; b)$.
\newblock {\em Annali di Matematica Pura ed Applicata}, 146(1):65--96, 1986.

\bibitem{sivashinsky1977nonlinear}
G.~I.~Sivashinsky.
\newblock Nonlinear analysis of hydrodynamic instability in laminar flames—i.
  derivation of basic equations.
\newblock {\em Acta Astronautica}, 4(11):1177--1206, 1977.

\bibitem{sivashinsky1980flame}
G.~I.~Sivashinsky.
\newblock On flame propagation under conditions of stoichiometry.
\newblock {\em SIAM J. Appl. Math.}, 39(1):67--82, 1980.

\bibitem{stein1970singular}
E.~Stein.
\newblock {\em {S}ingular integrals and differentiability properties of
  functions}.
\newblock Princeton University Press, 1970.

\bibitem{taylor} M.~E.~Taylor
\newblock \emph{Partial Differential Equations III}
\newblock Springer-Verlag, New York, 1996.

\bibitem{temambook}
R.~Temam.
\newblock {\em Infinite-dimensional dynamical systems in mechanics and
  physics}, volume~68 of {\em Applied Mathematical Sciences}.
\newblock Springer-Verlag, New York, 1988.

\bibitem{Mischaikowrigorous}
P.~Zgliczy{\'n}ski and K.~Mischaikow.
\newblock Rigorous numerics for partial differential equations: the
  {K}uramoto-{S}ivashinsky equation.
\newblock {\em Found. Comput. Math.}, 1(3):255--288, 2001.

\end{thebibliography}

\end{document}